\documentclass[reqno]{amsart}
\usepackage[utf8]{inputenc}

\usepackage{amsbsy}
\usepackage{amsfonts}
\usepackage{amsmath}

\usepackage{amsthm}
\usepackage{amssymb}
\usepackage{enumitem}
\usepackage{ytableau}

\usepackage{hyperref}
\usepackage{color}
\usepackage{mwe}
\usepackage{float}
\usepackage{diagbox}

\usepackage{etoolbox}

\theoremstyle{plain}
\newcounter{statements}
\numberwithin{statements}{section}
\numberwithin{equation}{section}

\newtheorem{definition}[statements]{Definition}
\newtheorem{theorem}[statements]{Theorem}
\newtheorem{lemma}[statements]{Lemma}

\newtheorem{conj}[statements]{Conjecture}

\newtheorem{corollary}[statements]{Corollary}
\newtheorem{prop}[statements]{Proposition}

\newtheorem{innercustomlemma}{Lemma}
\newenvironment{customlemma}[1]
  {\renewcommand\theinnercustomlemma{#1}\innercustomlemma}
  {\endinnercustomlemma}

\usepackage{xcolor}
\usepackage{graphicx}
\usepackage{subcaption}
\usepackage{bbm}
\usepackage{diagbox}
\usepackage{dsfont}
\graphicspath{ {images/} }
\usepackage{makecell}
\usepackage{float}
\usepackage{enumitem}
\usepackage[at]{easylist}
\usepackage{stackrel}
\usepackage[T1]{fontenc}
\usepackage{esint}
\usepackage[utf8]{inputenc}
\usepackage{geometry}
\usepackage{verbatim}

\usepackage{thmtools}
\usepackage{thm-restate}

\newcommand{\Mod}[1]{\ (\text{mod}\ #1)}

\newcommand{\cR}{\mathbb{R}}
\newcommand{\cZ}{\mathbb{Z}}

\newcommand{\cN}{\mathbb{N}}

\newcommand{\EE}{\mathbb{E}}
\newcommand{\PP}{\mathbb{P}}

\newcommand{\var}{\mathrm{Var}}
\newcommand{\cov}{\mathrm{Cov}}

\newcommand{\totient}{\varphi}
\newcommand{\floor}[1]{\lfloor#1 \rfloor}

\newcommand{\bigo}[1]{\mathcal{O}\left(#1\right)}
\newcommand{\sP}{\mathcal{P}}

\newcommand{\Li}{{ \rm Li}}
\newcommand{\prefactor}{{c_{Q}}}
\newcommand{\paren}[1]{\left(#1\right)}
\newcommand{\set}[1]{\left\{#1\right\}}
\newcommand{\abs}[1]{\left|#1\right|}
\newcommand{\mc}[1]{\mathfrak{C}_{#1}} 
\newcommand{\sound}{Soundararajan}
\newcommand{\cramer}{Cram\'er}

\newcommand{\gapcut}[2]{[#1]_{#2}}

\newcommand{\GeoApproxLemma}{Fix an integer $Q \geq 2$ and constant $c > 0$. Let $m$ be a non-negative integer. Let $u_i$ denote the $i\textsuperscript{th}$ smallest positive integer co-prime to $Q$. 
Define
$$T_1^m(n) = \sum_{k > 0}\frac{\paren{u_{n + k} - u_n}^m}{\log u_n}\left(1 - \frac{c}{\log u_n}\right)^{k - 1}$$
$$\quad\quad T_2^m(n) = \sum_{k > 0}\left(\frac{\paren{u_{n+k} - u_n}^m}{\log u_{n+k}}\prod_{0 < j < k}\left(1 - \frac{c}{\log u_{n+j}}\right)\right)$$.
Then
$$
  T_2^m(n) = T_1^m(n) + \bigo{\frac{\log(n)^{m + \varepsilon}}{n}}
$$
for any fixed $\varepsilon > 0$ as $n$ tends to infinity.} 

\begin{document}
\title[Prime Running Functions]
{ Prime Running Functions }
\author{Jaeyoon Kim }
\address{Department of Mathematics, University of Michigan, Ann Arbor, MI 48109-1043}
\email{jaeykim@umich.edu}

\begin{abstract}
We study  arithmetic functions $\Phi(x;d,a)$, called prime running functions, whose value at $x$  sums  the gaps between primes $p_k \equiv a \Mod d$ below $x$ and the next following prime $p_{k+1}$, up to $x$.
 (The following prime $p_{k+1}$ may  be in any residue class $\Mod d$.) 
We empirically observe systematic biases of order $x / \log x$ in $\Phi(x;d,a) - \Phi(x;d,b)$ for different $a,b$. We formulate modified \cramer{} models for primes and show that the corresponding sum of prime gap statistics
exhibits systematic biases of this order of magnitude. 
The predictions of such modified \cramer{} models are compared with the experimental  data.
\end{abstract}

\maketitle
\section{Introduction}

This paper studies a new class of prime counting statistics 
based on the size of gaps between primes, where the smaller prime in the gap is restricted to a fixed arithmetic progression.
The {\em prime running function} $\Phi(x; d, a)$ counts the number of integers $n \le x$ having
the property that the largest prime  $p \le n$ satisfies $p \equiv \, a \, (\bmod \, d)$.
Alternatively, these statistics may
be thought of as counting the primes in a fixed arithmetic progression, each  weighted
by the length of the gap from that prime to the next larger prime. 
We present experimental evidence  that
\begin{equation}\label{eq:main1}
\Phi(x; d, a)= \frac{1}{\varphi(d)} x + R(d; a) \frac{x}{\log x} + o( \frac{x}{\log x}),
\end{equation} 
may hold as $x \to \infty$ (Conjecture 2.3). 
In this formula,
even the main term 
$\Phi(x, d,a) \sim \frac{1}{\varphi(d)} x$ is conjectural  for $d \ge 3$  (Conjecture 2.1).
The main term is what one would expect from the mean of gap sizes not depending
on the modulus $a \, (\bmod \, d)$, while the term   $R(d; a)\frac{x}{\log x}$ 
quantifies  a ``bias term'' which is the main focus of this paper.
We  rigorously analyze a probabilistic model (modified \cramer{} model having a preliminary
sieving on a modulus $Q$)
which predicts  a functional form  of   shape  \eqref{eq:main1}, with a bias term present. 
For small moduli $d$, we compare the model prediction for $R(d,a)$, taking $Q$
to be a large primorial, against empirical estimates for $R(d,a)$.

The bias  phenomenon was discovered  in 
  study of  `prime running races' $\Phi(x;d,a) - \Phi(x;d,b)$, 
between two different residue classes $a, b$ (with $(ab, d) =1$). 
Such races are analogous to `prime number races' $\pi(x;d,a) - \pi(x; d, b)$,
on which there has been a large amount of work (see Sect. \ref{subsec:related_work}).
We present  evidence that prime running races have biases asymptotically equivalent to 
$Cx/ \log x$ for some constant $C= C(d; a, b)$. The conjectured formula  \eqref{eq:main1} 
above would give
$C(d; a, b)= R(d;a) - R(d;b)$.    
This bias phenomenon was discovered experimentally
for these statistics  by plotting the simultaneous movements of two prime
running races as $n$ increases on a single figure (Figure~\ref{fig:running}).
We plotted a walk on the square lattice $\cZ^2$ with $X$ component of the walk given by one prime running race and $Y$ component of the walk given by a different prime running race.
One can make similar plots for prime number races $\pi(x;d,a) - \pi(x; d, b)$. 
One sees a great
difference in the appearance of the plots in the two cases. 
The plots for prime number races resemble $2$-dimensional simple random walks, while
the plots for prime running races do not resemble random walks at all, and exhibit systematic biases
increasing with $x$. 
We illustrate this phenomenon with an example.

\subsection{Prime Walk}\label{subsec:11}
The following `prime walk' on the integer lattice $\cZ^2$ takes steps according to the location of the two different prime number races$\Mod 5$ as the variable $n$ increments. We begin the walk
from the origin $(0, 0)$ at time $n=1$. From there, we repeatedly increment $n$
by 1. Whenever $n=p_k$ is a prime, we do the following:
\begin{itemize}
\item if $p_k \equiv 1 \Mod 5 $, move down; add $(0,-1)$
\item if $p_k \equiv 2 \Mod 5 $, move left; add $(-1,0)$
\item if $p_k \equiv 3 \Mod 5 $, move up; add $(0,1)$
\item if $p_k \equiv 4 \Mod 5 $, move right; add $(1,0)$
\end{itemize}
If $n$ is not prime (or if $n = 5$), we do not move. 

\begin{figure}[H]
 \centering \includegraphics[width=0.6\textwidth]{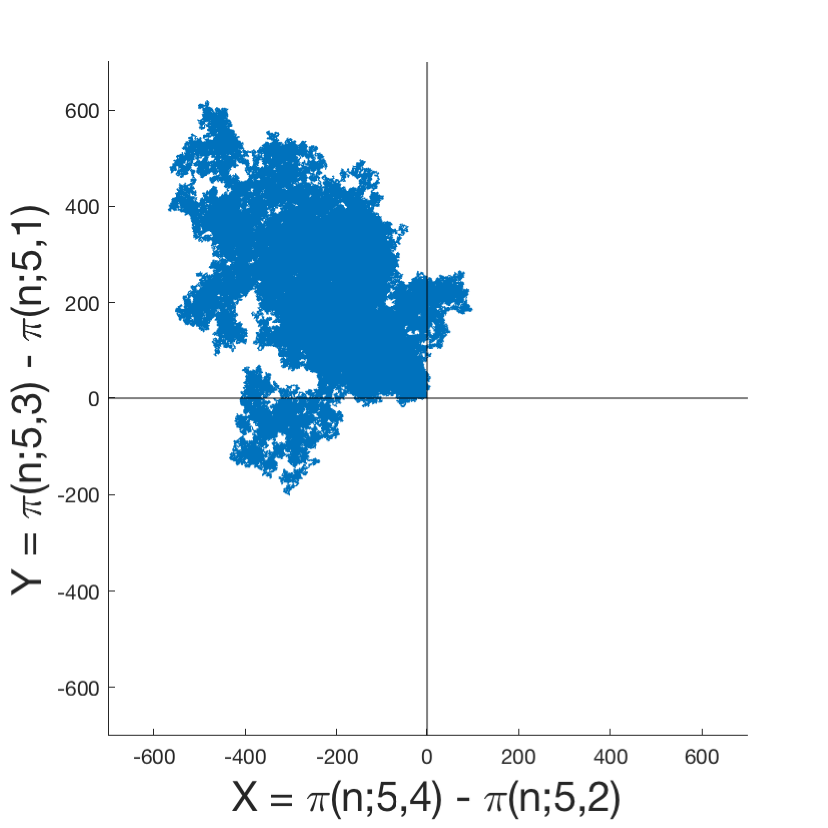}
 \caption{Plot of prime walk for $1 \leq n \leq 10^8$. } \label{fig:prime_walking}
\end{figure}

Figure~\ref{fig:prime_walking} presents the plot of points of the `prime walk' for $n \leq 10^8$.
The $n$-th point of the walk is located at position
$$(\pi(n; 5, 4) - \pi(n; 5, 2),\,\, \pi(n; 5, 3) - \pi(n; 5, 1))\quad \quad 1 \le n \le 10^8.$$
Using the terminology of Granville and Martin~\cite{prime_race}, Figure~\ref{fig:prime_walking}
exhibits the motion of two `prime number races'$\Mod 5$; the $Y$-component demonstrates the race between Team $3$ and Team $1$, while the $X$-component encodes the race
between Team $4$ and Team $2$. The resulting walk exhibits a slight Northwest bias with a maximum magnitude of order $10^3$. The Northwest bias is explained by Chebyshev's bias$\Mod 5$ (see Sect.~\ref{subsec:related_work}).
Qualitatively, figure ~\ref{fig:prime_walking}
resembles a sample path of a simple random
 walk, in that its maximum distance from the
origin is approximately proportional to the square root of the number of steps.

\subsection{Prime Run}
\label{sec:12}
We change the rules of the `prime walk' $(\bmod \, 5)$ above to obtain `\emph{prime run}'.
Whenever $n=p_k$ is prime, we move in the same direction as the prime walk. 
However, the prime run does not stop when $n$ is composite, it continues taking steps in the same direction that we were moving at time $n - 1$. 
Each time $n=p_k$ is prime, we have an opportunity for
changing directions. For the composite values of $n$ in between, we move in a straight
line at unit speed, following the previous direction. 

 To obtain the position when $n = p_{k+1} - 1$, we can apply the following algorithm to the position when $n = p_k - 1$.
\begin{itemize}
\item if $p_k \equiv 1 \Mod 5 $, move down until the next prime; add $(0,-(p_{k+1} - p_{k}))$
\item if $p_k \equiv 2 \Mod 5 $, move left until the next prime; add $(-(p_{k+1} - p_{k}),0)$
\item if $p_k \equiv 3 \Mod 5 $, move up until the next prime; add $(0,p_{k+1} - p_k)$
\item if $p_k \equiv 4 \Mod 5 $, move right until the next prime; add $(p_{k+1}-p_k,0)$
\end{itemize}
If $n= p_3= 5$, we stop the walk until 
the next prime $n=p_4=7$ is reached.
Instead of moving one step, the prime run increments by the magnitude of the gap between primes. 
Since the average gap size between the primes is $x / \pi(x) \sim \log(x)$, one might expect that the prime running plot will look approximately like the prime walk scaled up by a factor of $\log(x)$.

Figure~\ref{fig:running} presents the plot of points of the prime run for $n \leq 10^8$

\begin{figure}[H]
 \centering \includegraphics[width=0.6\textwidth]{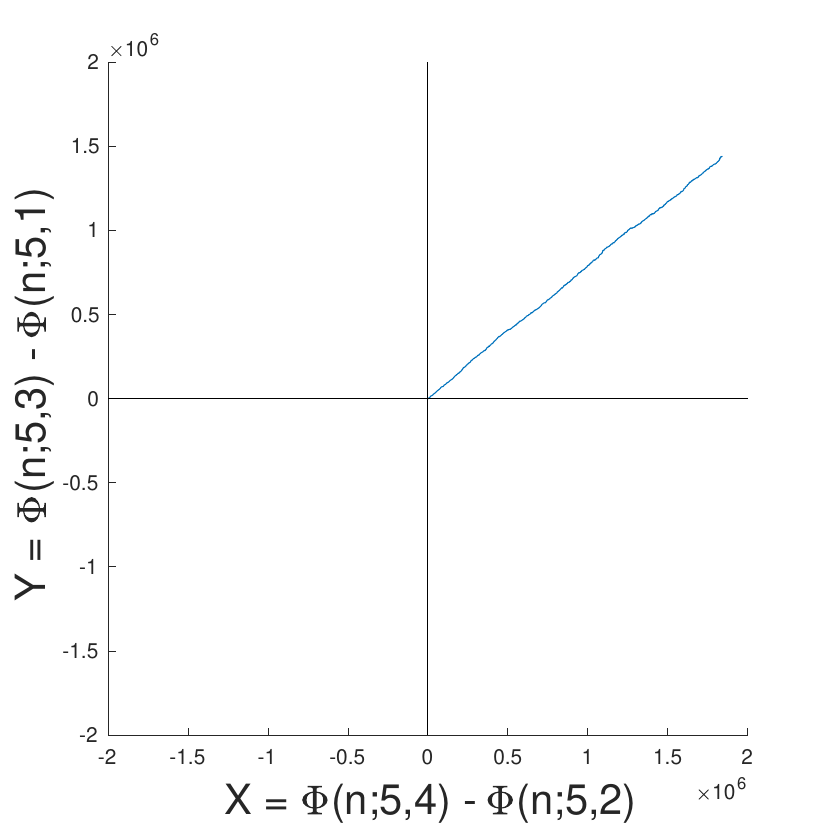}
 \caption{Plot of prime run for $1 \leq n \leq 10^8$.  }\label{fig:running}
\end{figure}

It looks like a line! Also, we observe that the maximum distance reached 
away from the origin is of order $10^6$, which is much larger than the $10^3$ spread for the prime walk.
We observe that the distance of order $10^6$ from the origin reached is considerably smaller than the $10^8$ steps taken, indicating that
the line in the plot has some thickness. Another observation is that the direction of drift in Figure~\ref{fig:running} is different from the direction of the `Chebyshev bias' in the prime walk shown in Figure~\ref{fig:prime_walking}.
Experimentally, this plot of the Prime Run exhibits a much
larger and more sharply focused drift than the drift in the prime walk.

\subsection{Related Work} \label{subsec:related_work}
The study of differences between the number of primes in different residue classes below a threshold $x$ has a long history. 
In the paper ``Comparative Prime Number Theory'' by Knapowski and Turan \cite[Problem 8]{KT62}, the study of $\pi(x;d,a) - \pi(x;d,b)$
was termed the 
(Shanks-Renyi) `prime number race'.
Let $\sP =\{ p_1 < p_2 < \ldots\} $ denote the set of primes, with $p_1=2, p_2=3$ etc. 
We recall that the counting function for primes in arithmetic progression
$a\, (\bmod \, d)$ is
\begin{equation}
  \pi(x;d,a) = \sum_{\substack{p_k \leq x \\p_k \equiv a \Mod d}} 1
\end{equation}
We assume $(a, d)=1$,
so that there are infinitely many primes in the class by Dirichlet's theorem. 

The subject of prime number races trace back to an assertion of Chebyshev \cite{chb1853} in 1853 (without proof) that
\begin{equation}
    \lim_{c \to 0^+}\sum_{n = 1}^\infty (-1)^{\frac{p_n + 1}{2}} e^{-p_n c} = +\infty.
\end{equation}
which gave a sense in which there are more primes of the form $4n + 3$ than of the form $4n + 1$.
In 1916 Hardy and Littlewood \cite{HardyLittlewood} (pg 141 - 148) proved Chebyshev's assertion under the assumption that the Riemann Hypothesis holds for $L(s, \chi_{-4})$.

However, already in 1914 Littlewood \cite{Litt1914} proved that $\pi(x;4,3) - \pi(x;4,1)$ has infinitely many sign changes. In 1995, by assuming the generalized Riemann hypothesis, Kaczorowski \cite{Kacz95} extended Littlewood's result to races between all
pairs of distinct nonzero residue classes$\Mod 5$.
It is now known that 
the lead of many prime races $\pi(x; d, a) - \pi(x; d, b)$ changes 
infinitely many times for many particular pairs of distinct reduced residue classes $a, b$
for many moduli $d$. For a survey on the case of prime moduli $d$, see Granville and Martin~\cite{prime_race}. For a general discussion of the distribution of the primes over different arithmetic progressions, see  Kaczorowoski \cite{KaczSurvey}.

 In 1994, Rubinstein and Sarnak \cite{rs94} introduced another variant of prime number races whiches quantifies the degree to which one race is ahead of another.
 Their framework is to measure
 the set of values of $x$ in which one member of a prime number race is ahead of another 
 using \textit{logarithmic density}. A set $S$ of positive integers has a well-defined logarithmic density
 $d(S)$ if the following limit exists:
 $$
 d(S) := \lim_{x \to \infty} \frac{1}{\log x} \left(\sum_{ \{ n \in S: n\le x\}} \frac{1}{n}\right).
 $$
 Rubinstein and Sarnak showed, assuming strong conjectures on
the distribution of zeros of $L$-functions, that a 
 logarithmic density exists for the set of $x$ 
such that $\pi(x; d, a) > \pi(x; d, b)$, where
$a$ and $b$ are residues$\Mod d$ having $(ab, d) =1$. 
Their analysis predicted that the logarithmic density of $x$ for which $\pi(x; 4, 3) > \pi(x; 4,1)$
is approximately $0.9959$. 
Rubinstein and Sarnak termed this phenomenon ``Chebyshev's bias''. 
See ~Feuerverger and Martin \cite{FM00} and Fiorilli \cite{FM13} for other examples of large biases in this sense. 

 The quantitative sizes of how far one member of a prime number
 race can be ahead of another (of such ``Chebyshev biases'') is 
 always small compared to the average value of these functions separately,
 which is about $\frac{1}{\totient(d)} \frac{x}{\log x}$.
The Prime Number Theorem for arithmetic progressions (\cite[Corollary 11.21]{MV07} and \cite{DavM00}) with $(a, d)=1$ states
 \begin{equation} \label{eqn:PNTAP}
\pi(x; d, a) = \frac{1}{\totient(d)} \Li(x) + \bigo{x e^{-c_d \sqrt{\log x}}},
\end{equation}
where $\Li(x)$ denotes the logarithmic integral $\Li(x) = \int_2^x\frac{dt}{\log t}$
and $c_d$ is some positive constant depending on $d$.
Then each prime number race$\Mod d$ with $\gcd(ab, d) =1$ satisfies
$$
|\pi(x; d, a) - \pi(x; d, b)| = \bigo{x e^{-c_d \sqrt{\log x}}}.
$$
Assuming the generalized Riemann hypothesis, this bound can be improved to 
$$|\pi(x; d, a) - \pi(x; d, b)| = \bigo{ x^{ \frac{1}{2} + \epsilon}}\quad \text{for any }\epsilon > 0.$$

In 2016, Lemke Oliver and \sound{} \cite{LOS16} introduced 
new prime statistics having ``unexpected biases'' which are quantitatively
very large as a function of $x$. 
These statistics concerned the counts up to $x$ for  $r$-tuples 
of $r$ consecutive primes whose  residue classes$\Mod d$
are specified. 
Restricting to $r=2$, let $\pi(x; d, (a, b))$ count the number of 
primes $p_k \leq x$ such that $p_k \equiv a \Mod d$ and $p_{k+1} \equiv b \mod d$.
Here, we follow the standard notation that $p_k$ denotes the $k$\textsuperscript{th} smallest prime. 
We call such functions
`consecutive prime counting functions in arithmetic progressions'. Here, one expects
equidistribution of these counts as $x \to \infty$ in the sense that
$$\pi(x; d, (a, b)) \sim \frac{1}{\totient(d)^2} \frac{x}{\log x} \quad \mbox{as} \quad x \to \infty,$$
although such results remain conjectural. 
Lemke Oliver and \sound{} formulated precise conjectures
on the asymptotic growth of $\pi(x; d, (a ,b))$ which predicts that the size of the bias terms can be 
as large as $x\, \frac{\log \log x}{(\log x)^2}$. 
Their main conjecture implies that  differences of such functions
$$
\pi(x; d, (a_1, b_1)) - \pi(x;d, (a_2, b_2)),
$$
which we may call `consecutive prime number races',
 sometimes observe biases of order $x\,\frac{ \log\log x}{(\log x)^2}$.
Such a large systematic bias of the consecutive prime number races lead to a fixed sign
for all sufficiently large $x$, which implies that one function wins the race for all sufficiently large $x$.

As an example, their main conjecture predicts\footnote{We take $r=2$ and $\frac{1}{8} = \frac{1}{2\totient(5)}$ 
in their Main Conjecture, page. E4447} 
$$
\pi(x; 5, (1,2)) - \pi(x; 5, (1, 1)) = \frac{1}{8}\, x\frac{\log\log x}{(\log x)^2} + \bigo{ \frac{x}{(\log x)^2}},
$$
an assertion implying that this bias will be positive for all large $x$.  This bias term 
 is smaller than the growth rate of
 $\pi(x; d, a)$ by a multiplicative factor $\frac{\log\log x}{\log x}$. 
 
Unlike the functions studied by Lemke Oliver and \sound{} which require two or more arithmetic progression conditions to exhibit bias, the prime running functions can exhibit a large bias even if we only restrict to a single arithmetic progression, as in \eqref{eq:main1}.

\subsection{Contents.} Section \ref{sec:2} defines prime running functions and formulates conjectures regarding the asymptotic behavior of the prime running function. In Section \ref{sec:3}, we present empirical evidence for $d= 3, 4, 5, 7$ and $25$ which provided
 the original basis for
some of the conjectures formulated Section \ref{sec:2}. In Section \ref{sec:cramer}, we formulate probabilistic models for the
primes which may explain the large bias terms.
These probabilistic models are versions of
the Cram\'{e}r model of random primes, modified by first making initial sieving to remove any integers not co-prime to sieve modulus $Q$. 
These models predict that the prime running functions observe a bias of order $x / \log x$ (Theorem~\ref{thm:MC-main} and Theorem~\ref{thm:variance}) and other behaviors (Theorem \ref{thm:anti_symmetry} and Theorem \ref{thm:radical}). These models provide heuristic justification for the conjectures made in Section~\ref{sec:2}.
The proof of Theorem~\ref{thm:MC-main} is found in Section~\ref{sec:cramer_expected_val}.
Section \ref{sec:bias_compute} provides an efficient method for computing the predicted bias computation by the model. The predictions of the \cramer{} model is compared with empirical data.
Section \ref{sec:conclusion} makes concluding remarks on analyzing probabilistic models for prime running functions. 

\section{ Prime Running Functions: Definitions and Conjectures}\label{sec:2}
\subsection{Prime Running Functions}\label{sec:21}

Now we introduce the prime running function. 

\begin{definition} \label{eq:running}
{\em
For $a \Mod d$, we define the \emph{Prime Running Function} as 
$$\Phi(x;d,a) = \sum_{\substack{1\leq n \leq x\\ \floor{n}_\sP \equiv a \Mod d}} 1.$$
Here the $\sP$-floor function $\floor{n}_{\sP}$ gives the 
largest prime less than or equal to $n$. We define $\floor{1}_\sP = 0$.
}
\end{definition}

The prime running function is similar to the prime counting function ``weighted'' by the magnitude of the prime gaps.
\begin{equation}\label{eq:gap_sum}
\Phi(x;d,a) = \sum_{\substack{p_{k+1} \leq x \\p_k \equiv a \Mod d}} \left(p_{k+1} - p_k\right) + e(x;d,a),
\end{equation}
where 
\[e(x;d,a) = 
\begin{cases}
\floor{x} - \floor{x}_\sP + 1 & \text{if } \floor{x}_\sP \equiv a \Mod d\\
0 & \text{otherwise}
\end{cases}
\]
The additional error term $e(x;d,a)$ is bounded by
$$|e(x; d, a) | = \bigo{x^{7/12+\epsilon}} $$
(see Huxley \cite[Chap. 28]{Hux72}).

 The plot of the Prime Run given in Figure~\ref{fig:running} is a plot of two differences
 of prime running functions 
 \begin{equation*}
    \paren{x_n, y_n} = \paren{\Phi(n;5,4) - \Phi(n;5,2), \Phi(n;5,3) - \Phi(n;5,1)}
 \end{equation*}
 for $1 \leq n \leq 10^8$. 

\subsection{Conjectures for Prime Running Functions}\label{sec:23}

It is natural to expect that the values of the prime running function are equidistributed among residue classes with $\gcd(a, d)=1$.

\begin{conj}\label{conj:main_term}
{\rm (Prime Running Function Main Term )}
For any integer $d \geq 2$ and any reduced residue $a \Mod d$,
$$\Phi(x;d,a) \sim \frac{1}{\totient(d)}x \quad as\ x\to\infty$$
\end{conj}

Aside from the trivial exception $d=2$, 
there seem to be no results known to give unconditional asymptotic formulas for functions of this type.
Furthermore, there does not even seem to be any lower bounds of the form $\Phi(x; d, a) > c x$ with $c>0$.

Since the average spacing between primes is of order $\log x$, if it were known that the prime gap
size distribution $p_{k+1} - p_k$ was independent of its congruence class $a \,(\bmod \,d)$
of $p_k$ to an error $o(\log x)$ as $x \to \infty$, then Conjecture \ref{conj:main_term} would follow.

The main empirical observation of this paper is the (apparent) existence of large biases in the prime
running function away from the expected main term. We formulate a conjecture
 characterizing the bias of the prime running function between different residues.
 
\begin{conj}\label{conj:running_error}
{\rm (Prime Running Bias Conjecture)}
For any integer $d \geq 2$ and integer $a$ with $\gcd(a, d) = 1$, there exists a constant $R(d;a)$ such that 
$$\Phi(x;d,a) = \frac{1}{\totient(d)}x +R(d;a)\frac{x}{\log x} + o\left(\frac{x}{\log x}\right)$$
\end{conj}
The order of magnitude $x / \log x$ for the bias term in conjecture~\ref{conj:running_error} is predicted by a probabilistic model in Section~\ref{sec:cramer}.

Assuming Conjecture~\ref{conj:running_error}, by taking the differences of two prime running functions, we can directly observe the bias term:
$$
\Phi(x;d,a_1) - \Phi(x;d, a_2) = (R(d; a_1) - R(d; a_2)) \frac{x}{\log x} + o\left(\frac{x}{\log x}\right)
$$
In Section \ref{sec:3}, we present empirical estimates of the constants $R(d;a)$
for $d= 3, 5$ and $7$.
We call the constants $R(d; a),$ \emph{bias constants}.

The empirical data and a probability model (see Theorem~\ref{thm:anti_symmetry}) suggest that the following anti-symmetry property of the bias constants may hold.

\begin{conj}\label{conj:odd_running_error}
{\rm (Bias Constant Anti-symmetry Conjecture)}
The bias constants for prime running function for modulus $d$ satisfy
$$R(d;-a) = -R(d;a),$$
when $(a,d) = 1$.
\end{conj}
In addition, Conjecture~\ref{conj:running_error} for $d = 3$ implies anti-symmetry $R(3;1) = -R(3;2)$ since $\Phi(x;3,1) + \Phi(x;3,2) = x + \bigo{1}$.

Limited empirical data and a probabilistic model (see Theorem~\ref{thm:radical}) support the conjecture that the bias constants$\Mod d$ depend only on the square-free part $d_{sf}$ of $d$, also called the radical of $d$, see~\cite{diophantine}.
\begin{equation}
    d_{sf} = rad(d) := \prod_{p | d} p
\end{equation}
\begin{conj}\label{conj:radical}{\rm (Radical Equivalence Conjecture)}
For all $d \geq 2$, with $(a , d) = 1$,
\begin{equation}
    R(d;a) = \frac{\totient(d_{sf})}{\totient(d)}R(d_{sf}; a),
\end{equation}
where $d_{sf} = rad(d)$ is the square-free part of $d$. 
\end{conj}
In particular, $R(d;a) = R(d;a')$ if $a \equiv a' \Mod{d_{sf}}$. For special case $d = 2$, we know unconditionally that $R(2;1) = 0$. Thus conjecture~\ref{conj:radical} predicts that
\begin{equation}
    R(2^j;a) = 0
\end{equation}
for all $j \geq 1$ and $a \equiv 1 \Mod 2$.

\section{Experimental Results}\label{sec:3}
In this section, we present numerical data on the prime running function for a few small
modulus $d$ over their residue classes. In Section~\ref{sec:3_primes}, we provide data for $d = 3, 5$, and $7$. In Section~\ref{sec:3_powers}, we provide data for $d = 4$ and $25$.

\subsection{ Prime Running Function Data for Prime Modulus}\label{sec:3_primes}
We first present data on the prime running functions for prime values of $d$ and compare them to the predicted values from the main term Conjecture~(\ref{conj:main_term}). Table~\ref{tab:r_mod3} and Table~\ref{tab:r_mod5} give numerical data for $d = 3$ and $d = 5$ at $x = 10^8, 10^{10}, 10^{12}$.

\begin{table}[H]
\begin{center}
	\begin{tabular}{|c|c|c|c|} 
		\hline
		\multicolumn{4}{|c|}{$\Phi(x;3,a)$}\\
		\hline
     \diagbox{a}{x} & $x = 10^8$ & $x = 10^{10}$ & $x = 10^{12}$\\
    \hline
		a = 1 & 51209542 & 5091131912 & 507317304782\\
		\hline
		a = 2 & 48790455 & 4908868085 & 492682695215\\
		\hline
		Predicted & 50000000 & 5000000000 & 500000000000\\
		\hline
	\end{tabular}
	\caption{Value of the prime running function $\Phi(x;3,a)$ at different values of $x$ and $a \Mod 3$.}
	\end{center}
\end{table}

\begin{table}[H]
\begin{center}
	\begin{tabular}{|c|c|c|c|} 
		\hline
		\multicolumn{4}{|c|}{$\Phi(x;5,a)$}\\
		\hline
        \diagbox{a}{x} & $x = 10^8$ & $x = 10^{10}$ & $x = 10^{12}$\\
        \hline
		a = 1 & 24644198 & 2470292440 & 247456175258\\
		\hline
		a = 2 & 23714857 & 2401583475 & 241999191675\\
		\hline
		a = 3 & 26085716 & 2588759228 & 257451209200\\
		\hline
		a = 4 & 25555226 & 2539364854 & 253093423864\\
		\hline
		Predicted & 25000000 & 2500000000 & 250000000000\\
		\hline
	\end{tabular}
	\caption{Value of the prime running function $\Phi(x;5,a)$ at different values of $x$ and $a \Mod 5$.}
	\end{center}
\end{table}
This numerical data suggests that the main term is $\frac{1}{\totient(d)} x$ and that systematic bias error terms are present. 
 
The size of the bias appears to be growing more slowly than the main term $\frac{1}{\totient(d)}x$ as $x$ increases in powers of $10$.
 
To fit the data to conjecture ~\ref{conj:running_error}, we introduce a new function.
\begin{definition}\label{def:numeric_bias_const}
For integer $d \geq 2$ and reduced residue $a \Mod d$, we define the \emph{rescaled bias function} $R(x;d,a)$ by
\begin{equation}
R(x;d,a) := \left(\Phi(x;d,a) - \frac{1}{\totient(d)}x\right)\frac{\log x}{x}.
\end{equation}
\end{definition}
Conjecture~\ref{conj:running_error} can now be rewritten in the following form.
\begin{conj}
For all $d \geq 2$, with $\gcd(a,d) = 1$ the following limit exists. 
$$R(d;a) = \lim_{x\to\infty}R(x;d,a).$$
\end{conj}
 
 Figure ~\ref{fig:mod3_run_scaled} and Figure \ref{fig:mod5_run_scaled}
 plots the rescaled bias functions for $d = 3$ and $d = 5$ for $x \leq 10^{10}$. The resulting curves appear approximately flat, which supports the conjecture that the prime running functions approach $\frac{1}{\totient(d)}x + R(d; a) \frac{x}{\log x}$,
 where $R(d; a)$ is the bias constant.

\begin{figure}[H]
\centering
 \includegraphics[width=0.7\textwidth]{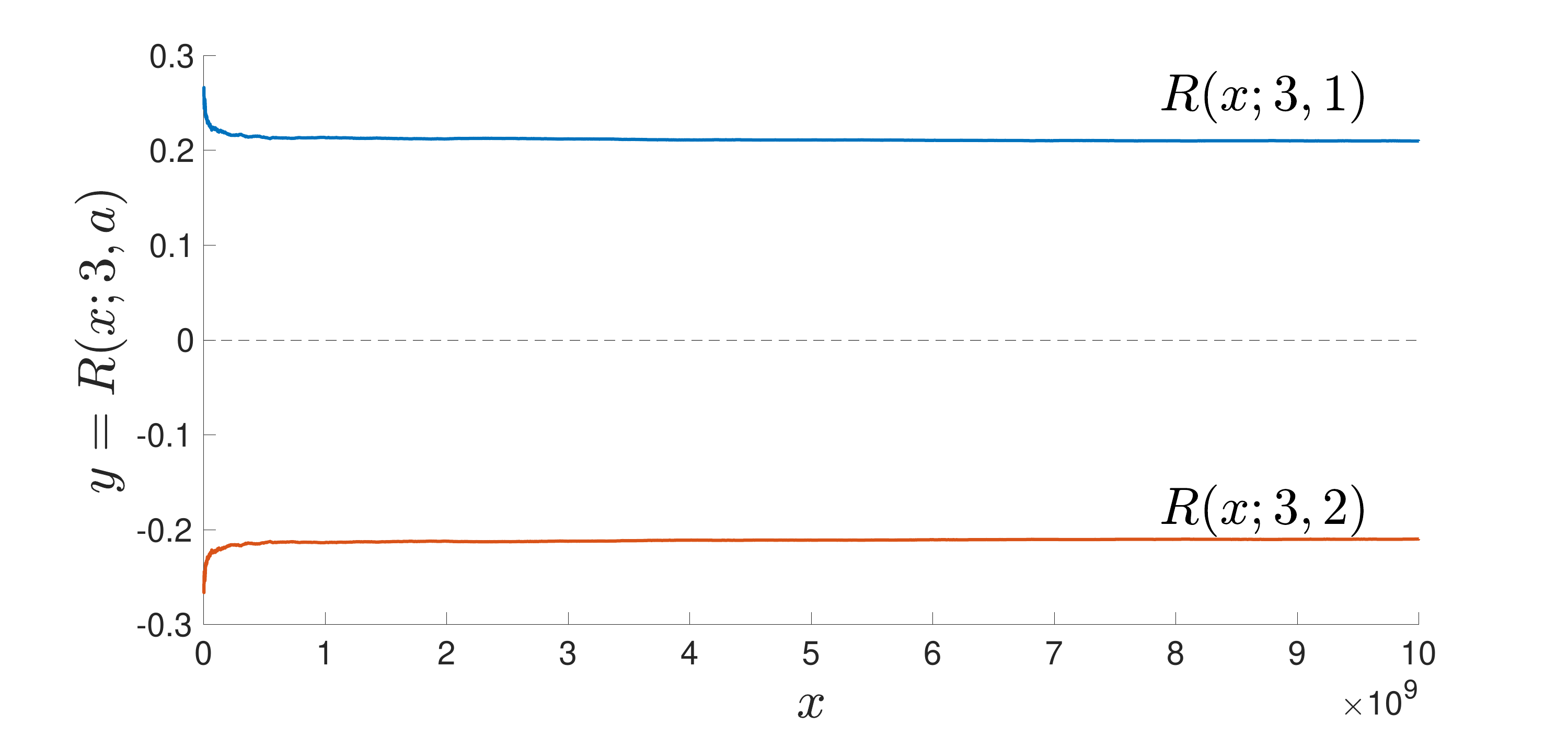}
\caption{Plot of $R(x;3,a)$ 
for all reduced residues $a \Mod 3$ and $x \leq 10^{10}$. The line $y = 0$ is marked with a dashed line.}\label{fig:mod3_run_scaled}
\end{figure}
 
\begin{figure}[H]
\centering
 \includegraphics[width=0.7\textwidth]{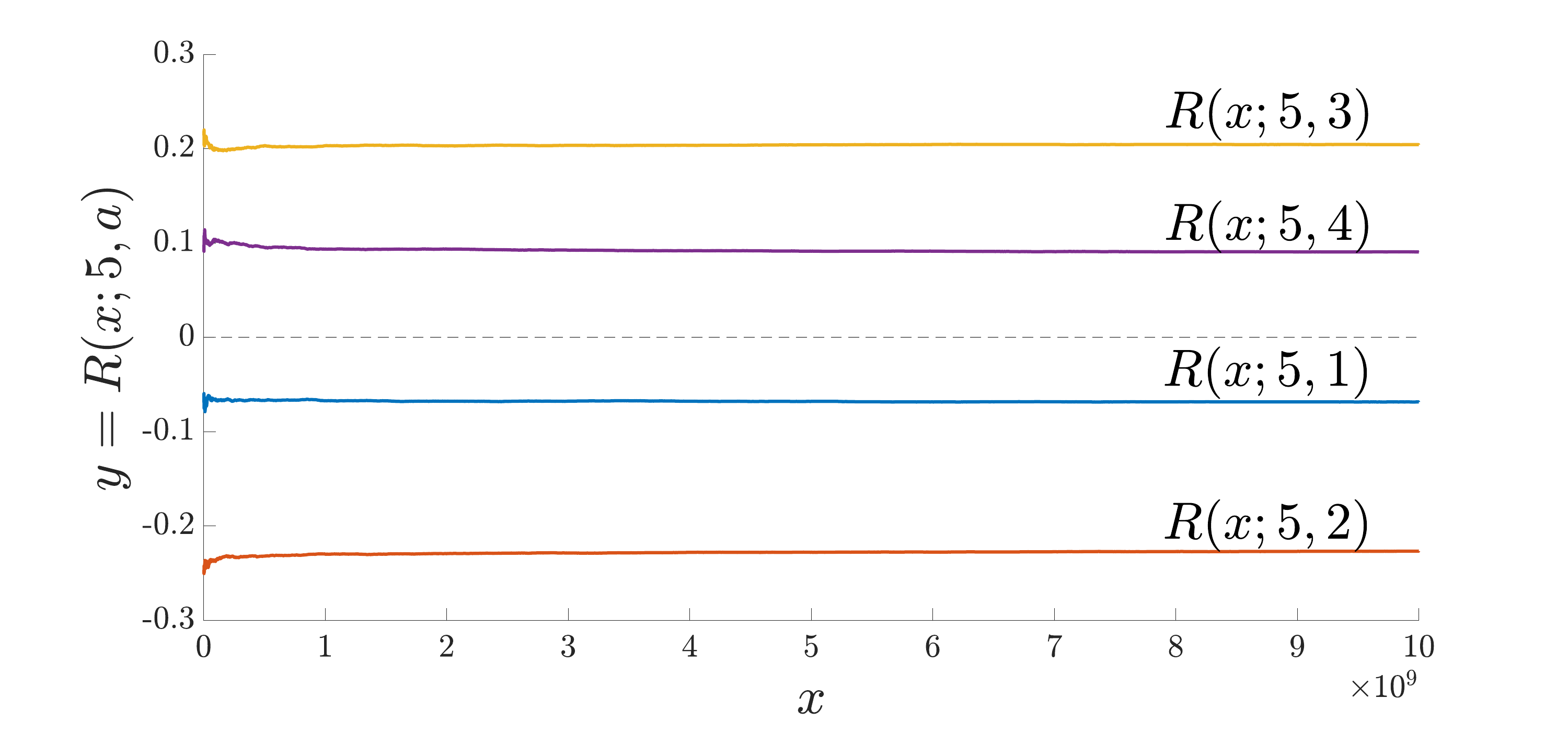}
 \caption{Plot of $R(x;5,a)$ 
for all reduced residues $a \Mod 5$ and $x \leq 10^{10}$. The line $y = 0$ is marked with a dashed line.} \label{fig:mod5_run_scaled}
\end{figure}

Tables \ref{tab:r_mod3}, \ref{tab:r_mod5} and \ref{tab:r_mod7} numerically computes the values of $R(x;d,a)$ for moduli $d= 3$, $d=5$ and $d=7$ at $x = 10^8, 10^{10}, 10^{12}$. 

\begin{table}[H]
\begin{center}
	\begin{tabular}{|c|c|c|c|} 
	\hline
	    \multicolumn{4}{|c|}{$R(x;3,a)$}\\
	\hline
        \diagbox{a}{x} & $x = 10^8$ & $x = 10^{10}$ & $x = 10^{12}$\\
    \hline
		1 & 0.2228 & 0.2098 &0.2022\\
	\hline
		2 & -0.2228 & -0.2098 &-0.2022\\
		\hline
	\end{tabular}
	\caption{Values of $R(x;3,a)$ for various values of $x$ and $a$.}
	\label{tab:r_mod3}
\end{center}
\end{table}

\begin{table}[H]
\begin{center}
	\begin{tabular}{|c|c|c|c|} 
	\hline
	    \multicolumn{4}{|c|}{$R(x;5,a)$}\\
	\hline
        \diagbox{a}{x} & $x = 10^8$ & $x = 10^{10}$ & $x = 10^{12}$\\
    \hline
		1 & -0.0655 & -0.0684 & -0.0703\\
	\hline
		2 & -0.2367 & -0.2266 & -0.2211\\
	\hline
        3 & 0.2000 & 0.2044 & 0.2059\\
	\hline
        4 & 0.1023 & 0.0906	& 0.0855\\
	\hline
	\end{tabular}
	\caption{Values of $R(x;5,a)$ for various values of $x$ and $a$.}
	\label{tab:r_mod5}
\end{center}
\end{table}

\begin{table}[H]
\begin{center}
	\begin{tabular}{|c|c|c|c|} 
	\hline
	    \multicolumn{4}{|c|}{$R(x;7,a)$}\\
	\hline
        \diagbox{a}{x} & $x = 10^8$ & $x = 10^{10}$ & $x = 10^{12}$\\
    \hline
		a = 1 & 0.1530 & 0.1501 & 0.1461\\
	\hline
	    a = 2 & -0.0780 & -0.0709 & -0.0680\\
	\hline
        a = 3 & 0.0588 & 0.0527 & 0.0506\\
	\hline
        a = 4 & -0.0681 & -0.0601	& -0.0571\\
	\hline
        a = 5 & 0.0583 & 0.0590	& 0.0626\\
	\hline
        a = 6 & -0.1240 & -0.1308 & -0.1343\\
	\hline
	\end{tabular}
	\caption{Values of $R(x;7,a)$ for various values of $x$ and $a$.}
	\label{tab:r_mod7}
\end{center}
\end{table}
In tables \ref{tab:r_mod3}, \ref{tab:r_mod5}, and \ref{tab:r_mod7}, slow trends are visble, but their directions (increase of decrease in magnitude) seems to vary with $a$. Furthermore, the data are consistent with the anti-symmetry Conjecture~\ref{conj:odd_running_error}.

\subsection{Prime Running Function Data for Prime Power Modulus}\label{sec:3_powers}
Figure~\ref{fig:mod4run_scaled} below plots $R(x;4,a)$ for $x \leq 10^{10}$.
There appears to be a smaller bias for the prime running functions for $1 \, (\bmod \, 4)$ and $3 \, (\bmod \, 4)$.
This is consistent with conjecture~\ref{conj:radical} which would imply that $R(4;1) = R(4;3) = 0$.

\begin{figure}[H]
\centering
\begin{subfigure}[b]{0.8\textwidth}
   \includegraphics[width=1\linewidth]{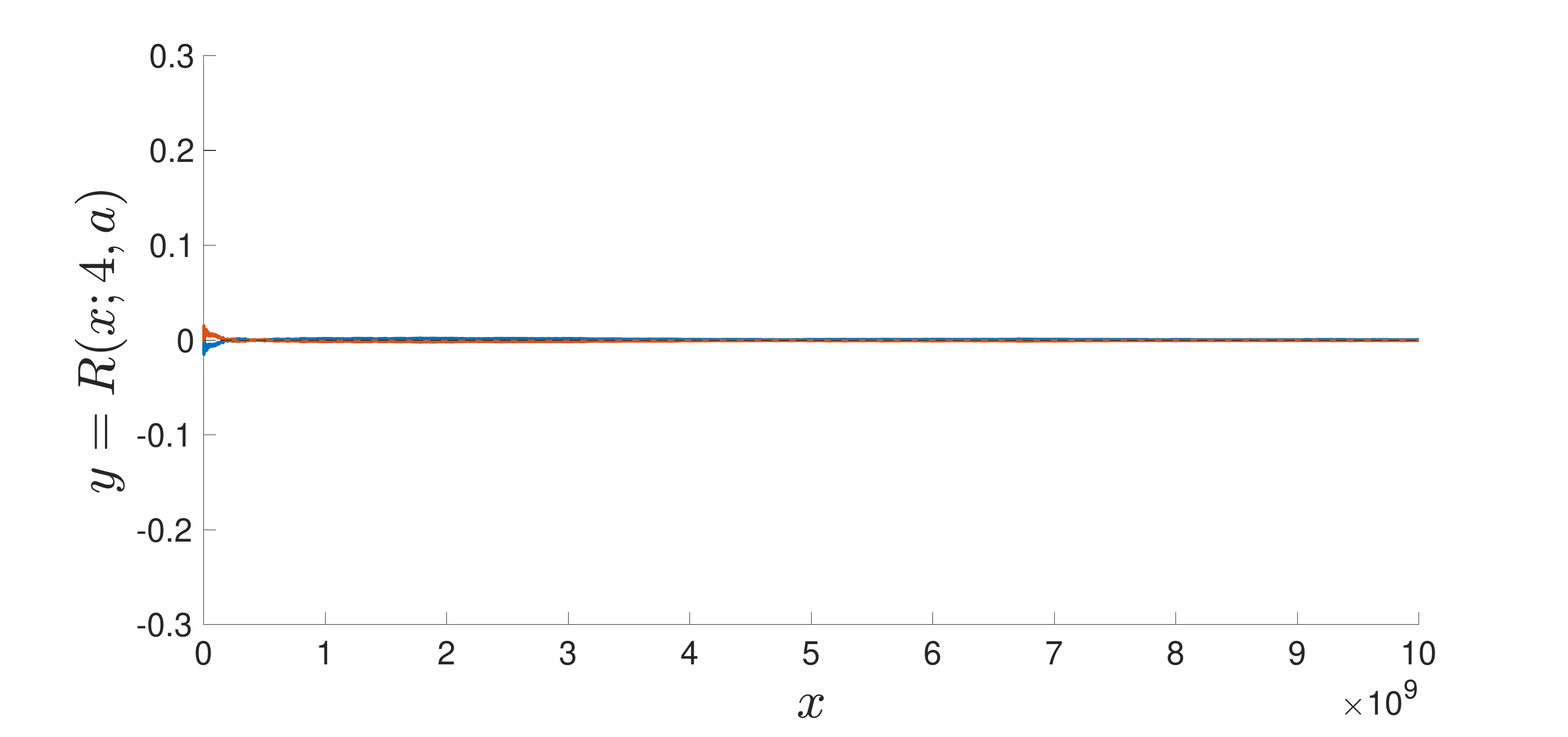}
   \caption{}
   \label{fig:Ng1} 
\end{subfigure}

\begin{subfigure}[b]{0.8\textwidth}
   \includegraphics[width=1\linewidth]{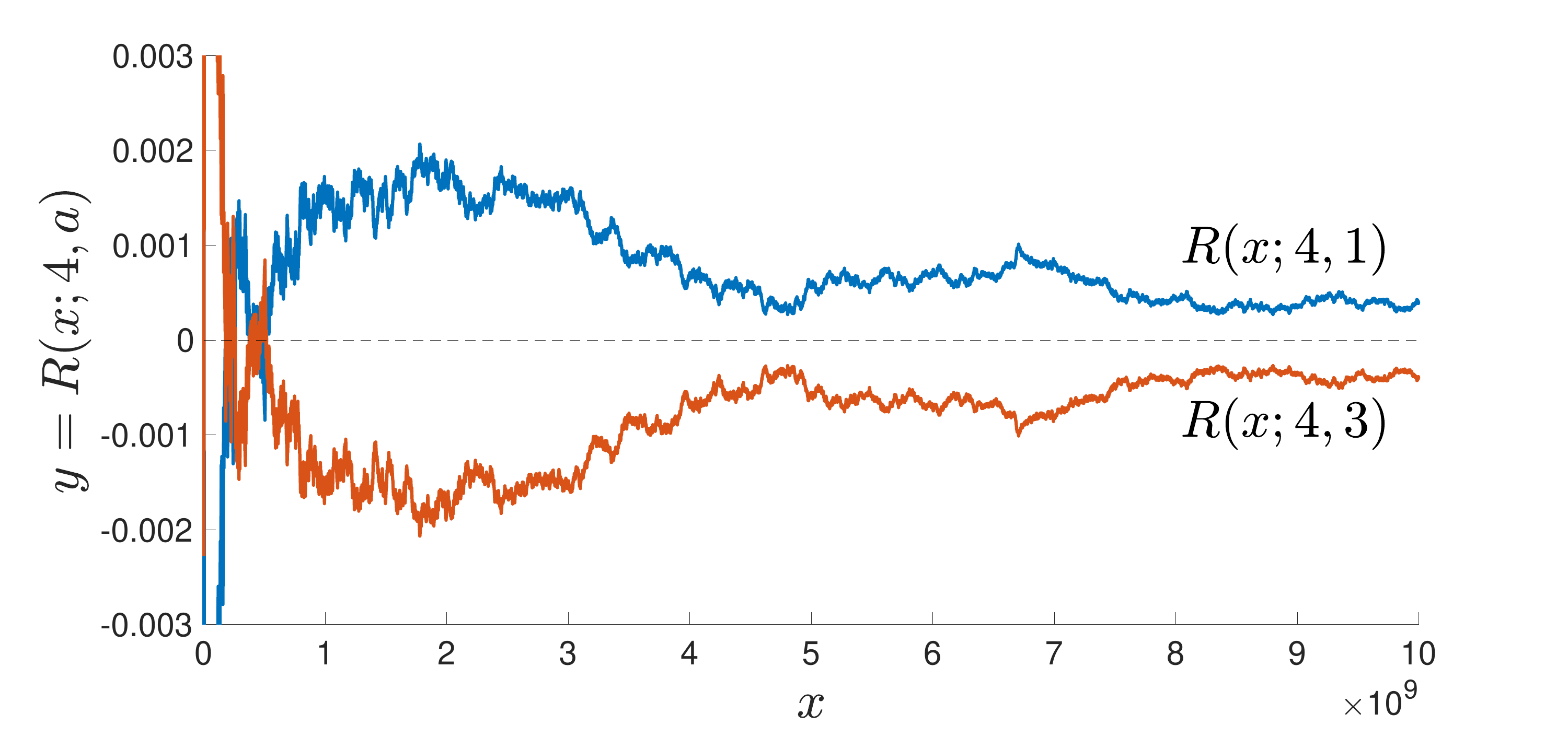}
   \caption{}
   \label{fig:Ng2}
\end{subfigure}

\caption[Two numerical solutions]{(a) Plot of $R(x;4,a)$ 
for all reduced residues $a \Mod 4$ and $x \leq 10^{10}$. The line $y = 0$ is marked with a dashed line. The axis are set to the same scale as Figures \ref{fig:mod3_run_scaled} and \ref{fig:mod5_run_scaled}.
(b) Y-axis is zoomed in by a scale of 100.}
\label{fig:mod4run_scaled}
\end{figure}

Figure~\ref{fig:mod4run_no_scale} presents the unscaled prime running race between $1 \Mod 4$ and $3 \Mod 4$. Chebyshev's bias for$\Mod 4$ is illustrated in Figure~\ref{fig:mod4_count_run_no_scale}. In the depicted domain, the sign of $\Phi(x;4,1) - \Phi(x;4,1)$ is predominantly positive, which is the opposite sign from Chebyshev's bias $\pi(x;4,1) - \pi(x;4,3)$.

We see that unlike prime races between prime moduli, the bias for is of much smaller order (roughly of order $\sqrt{x}$). We observe that for large values of $x$ in the plot, $\Phi(x;4,1) - \Phi(x;4,3) > 0$

\begin{figure}[H]
 \centering \includegraphics[width=0.7\textwidth]{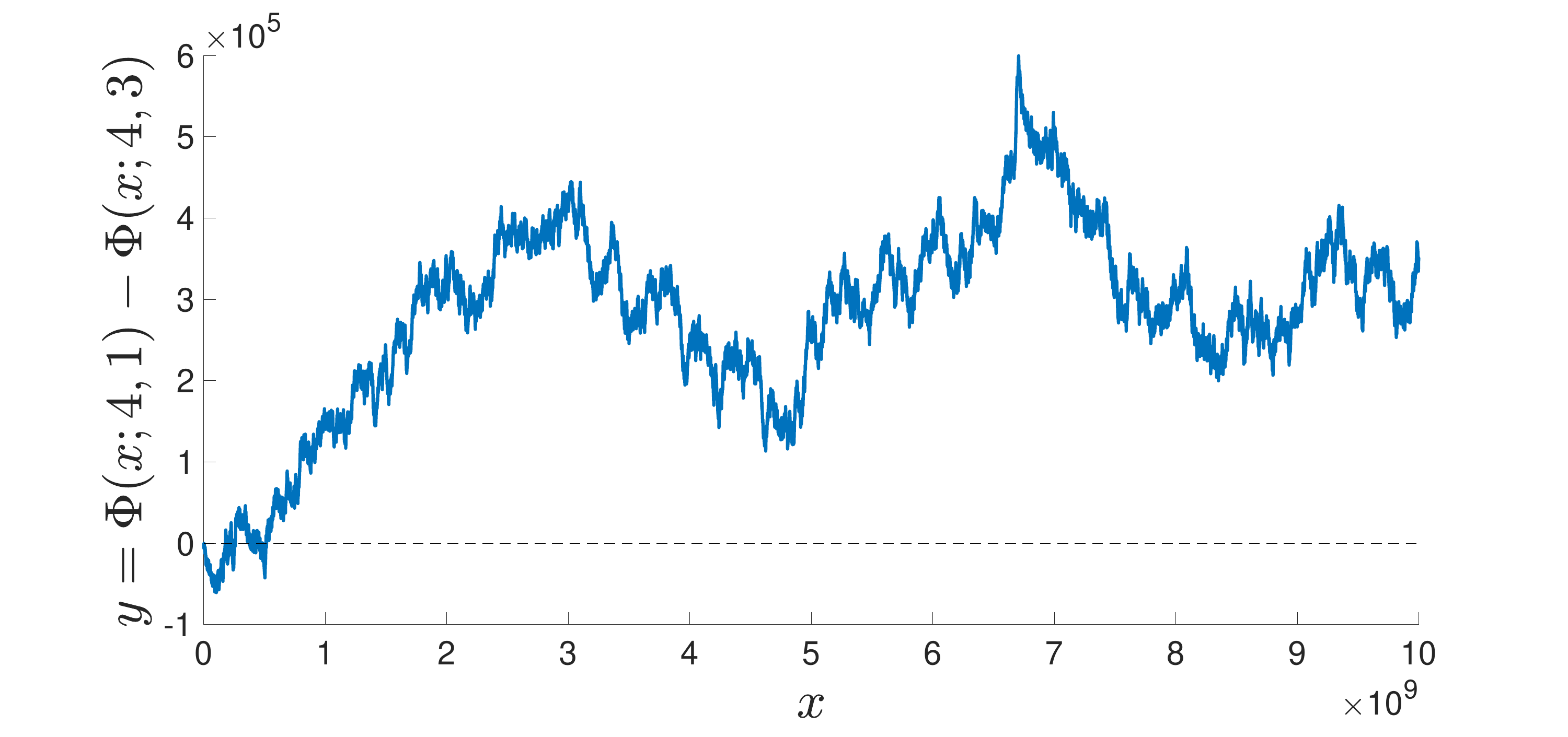}
 \caption{Plot of $\Phi(x;4,1) - \Phi(x;4,3)$ against $x$ for $x \leq 10^{10}$. The line $y = 0$ is marked with a dashed line. }\label{fig:mod4run_no_scale}
\end{figure}

\begin{figure}[H]
 \centering \includegraphics[width=0.7\textwidth]{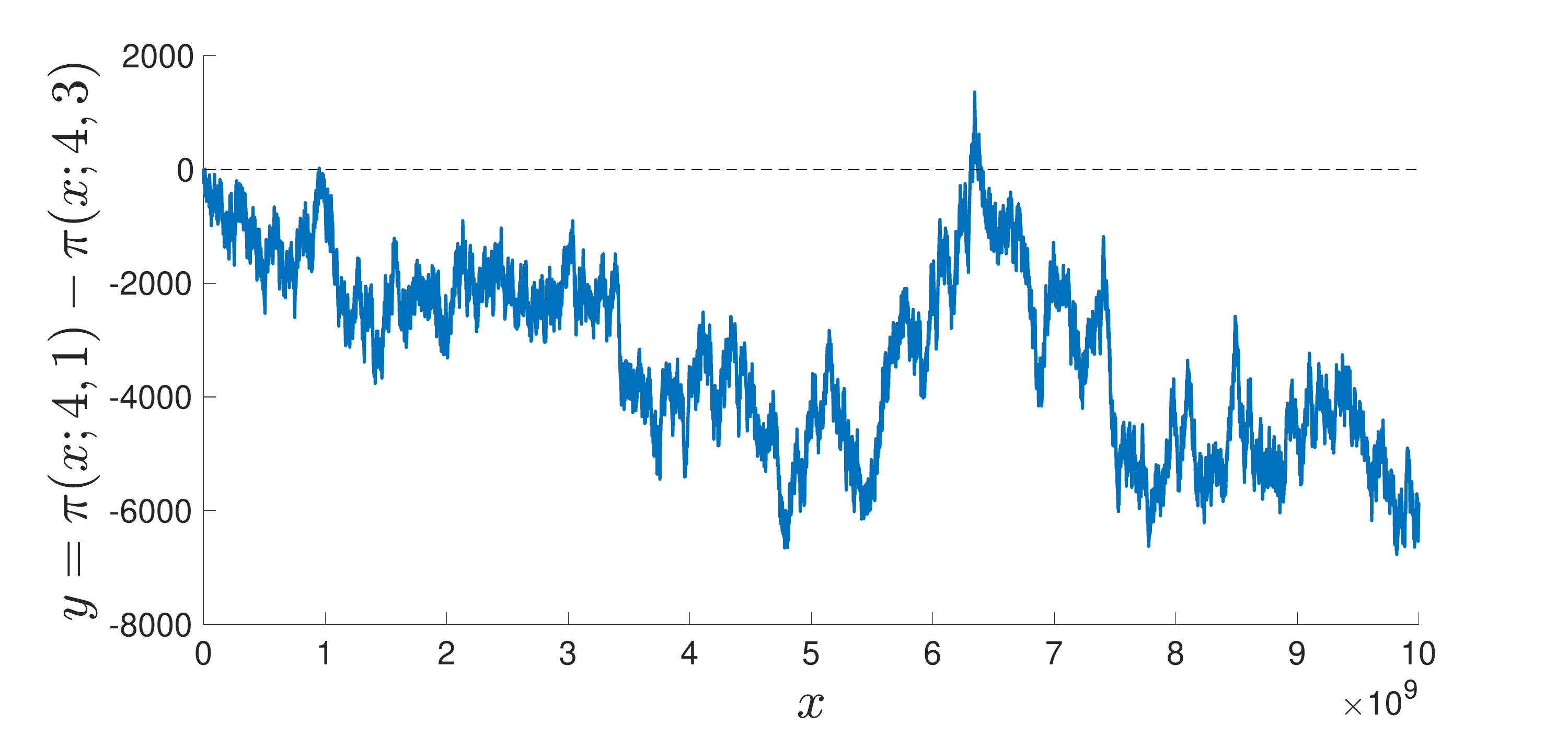}
 \caption{Plot of $\pi(x;4,1) - \pi(x;4,3)$ against $x$ for $x \leq 10^{10}$. The line $y = 0$ is marked with a dashed line. }\label{fig:mod4_count_run_no_scale}
\end{figure}
Now we present data on prime running functions $\Mod{25}$ as evidence for Conjecture~\ref{conj:radical}.

\begin{figure}[H]
\centering
 \includegraphics[width=0.7\textwidth]{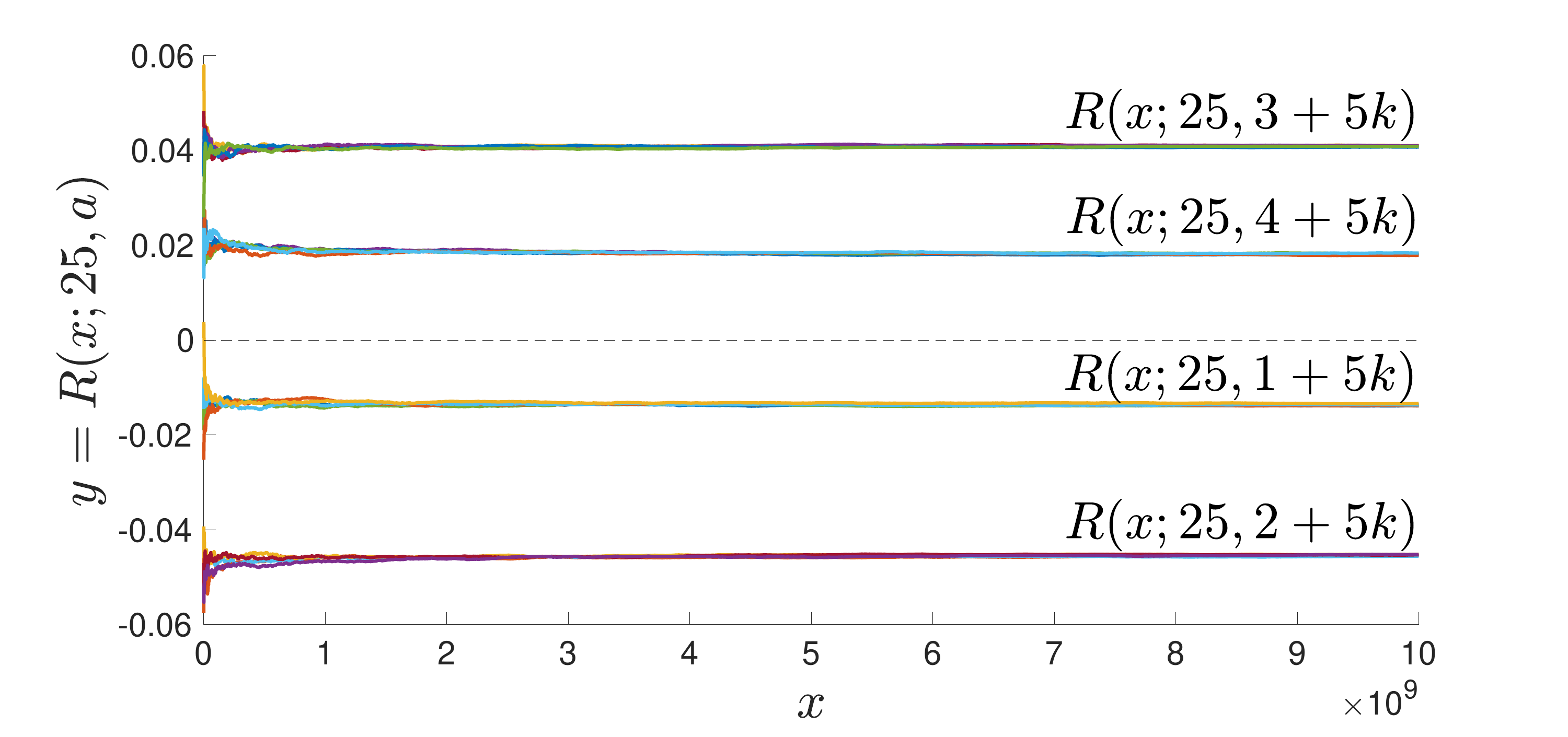}
 \caption{Plot of $R(x;25,a)$ 
for all reduced residues $a \Mod{25}$ and $x \leq 10^{10}$. The line $y = 0$ is marked with a dashed line. In the figure, one can see four ``solid lines''. However, each ``line'' is overlap of plots of $R(x;25,a)$ for five different values of $a$. The top ``line'' is composed of plots of $R(x;25,3)$, $R(x;25,8)$, $R(x;25,13)$, $R(x;25,18)$, and $R(x;25,23)$. The figure is scaled down by a factor of $5$ compared to figures \ref{fig:mod3_run_scaled} and \ref{fig:mod5_run_scaled}} \label{fig:mod25_run_scaled}
\end{figure}

 Table \ref{tab:r_mod4} below numerically computes the values of $R(x;4,a)$.
\begin{table}[H]
\begin{center}
	\begin{tabular}{|c|c|c|c|} 
	\hline
	    \multicolumn{4}{|c|}{$R(x;4,a)$}\\
    \hline
        \diagbox{a}{x} & $x = 10^8$ & $x = 10^{10}$ & $x = 10^{12}$\\
    \hline
	    1 & $-0.0041$ & $0.0004$ & $0.0002$\\
		\hline
		3 & $0.0041$ & $-0.0004$ & $-0.0002$\\
		\hline
	\end{tabular}
	\caption{Values of $R(x;4,a)$ for various values of $x$ and $a$.}
	\label{tab:r_mod4}
\end{center}
\end{table}
 Table \ref{tab:r_mod4} suggests that $R(4, 1)= R(4, 3) = 0$ as predicted by Conjecture \ref{conj:odd_running_error}.

\begin{table}[H]
\begin{center}
	\begin{tabular}{|c|c|c|c|} 
	\hline
	    \multicolumn{4}{|c|}{$R(x;25,a)$}\\
	\hline
        \diagbox{a}{x} & $x = 10^8$ & $x = 10^{10}$ & $x = 10^{12}$\\
    \hline
		a = 1 & $-0.0129$ & $-0.0139$ & $-0.0140$\\
	\hline
		a = 6 & $-0.0131$ & $-0.0136$ & $-0.0141$\\
	\hline
		a = 11 & $-0.0144$ & $-0.0139$ & $-0.0141$\\
	\hline
		a = 16 & $-0.0127$ & $-0.0137$ & $-0.0141$\\
	\hline
		a = 21 & $-0.0125$ & $-0.0134$ & $-0.0140$\\
	\hline
	\end{tabular}
	\caption{Values of $R(x;25,a)$ for various values of $x$ and $a \equiv 1 \Mod 5$.}
	\label{tab:r_mod25}
\end{center}
\end{table}
From Table \ref{tab:r_mod25}, it seems that values of $R(x;25;1 + 5k)$ for $k = 0,1,2,3,4$ become closer as the value of $x$ increases. This behavior is consistent with Conjecture \ref{conj:radical}.

\section{Probabilistic Models for Bias Terms in Prime Running Functions}\label{sec:cramer}
We study probabilistic models for ``random primes'' which can model prime gaps and prime running functions$\Mod d$.
We show that modified Cram\'{e}r models (defined in Section \ref{sec:51}) produce bias terms of order $x / \log x$ associated
with the prime running functions. 

\subsection{Modified \cramer{} Models}\label{sec:51}
The original probabilistic model of Cram\'{e}r (\cite{cr35}, \cite{cr36}) 
picks independently for each integer $n \ge 3$ to be ``$\mc{}$-prime" with
probability $\frac{1}{\log n}$. 
The \cramer{} model seems to accurately predict
many statistics on primes. For example, the \cramer{} model predicts that $|\pi(x) - \Li(x)|$ lies within the predicted range by the Riemann hypothesis.
 However, it does not account for arithmetic restrictions
on prime gaps and primes in arithmetic progressions. For example, almost all sample sequences of $\mc{}$-primes
contain infinitely many gaps of size $1$ between consecutive $\mc{}$-primes and contain infinitely many even numbers as $\mc{}$-primes.

We study a modified version of the \cramer{} model
for the distribution of primes, that imposes initial sieving by an integers $Q \geq 2$ called the \emph{sieve modulus}, followed by a probability model imposed on
the unsieved elements. The initial sieving builds in arithmetic restrictions.
In this model, we let integer $n$ with $\gcd(n, Q) = 1$ be a ``$\mc{Q}$-prime" with probability
 $\frac{\prefactor}{\log n}$ 
where $\prefactor$ is the pre-factor
\begin{equation}\label{eq:def_prefactor}
\prefactor := \frac{Q}{\totient(Q)} =\prod_{p | Q} \left(1+ \frac{1}{p-1}\right).
\end{equation}
The pre-factor $\prefactor$ quantifies the increased chance to be prime after the initial sieving. 
Modifications of \cramer{} models that make such an  initial sieving were suggested in 1995 by 
Granville \cite{granville_cramer95}. They were later studied
 by Pintz \cite{pintz07}.
 
Formally, for fixed integer $Q \geq 2$, we define a sequence of independent Bernoulli random variables $Z_{n,Q}$ by
\begin{equation}\label{eq:rand_var}
\Pr[Z_{n,Q} = 1] = 
\begin{cases}
\frac{\prefactor}{\log n} & gcd(n,Q) = 1\\
0 & gcd(n,Q) \ne 1\\
\end{cases}
\end{equation}
If $\frac{\prefactor}{\log n}$ from \eqref{eq:rand_var} exceeds $1$, 
then we replace it by $1$, a change that affects only finitely many values of $n$.
 If $Z_{n,Q} = 1$ then
 we say that $n$ is a {\em $\mc{Q}$- prime}.

In the modified \cramer{} model, we can define a random variable version
of the prime running functions for these moduli $d$ that divide the sieve modulus $Q$.

\begin{definition}\label{def:cond_gap}
{\em
The {\em conditional gap} $W_{n,Q}$ is a random variable defined as a function of random variables $Z_{n,Q}, Z_{n + 1, Q}, \ldots$
\begin{equation}
W_{n, Q} :=
\begin{cases}
m - n & \mbox{if} \,\, Z_{n,Q} = 1\, \mbox{and} \, Z_{n+1,Q} = Z_{n+2,Q} = \cdots = Z_{m-1,Q} = 0\,\, \mbox{and} \,\, Z_{m,Q} = 1.\\
0 & \text{otherwise}
\end{cases}
\end{equation}

We call $W_{n,Q}$ the {\em conditional gap} because if $n$ is a $\mc{Q}$-prime, then the value of $W_{n,Q}$ will equal the difference between $n$ and the next $\mc{Q}$-prime.
}
\end{definition}

\begin{definition}
(Random Prime Running Function)
{\em Let $Q \geq 2$ be an integer divisible by $d$. 
For fixed $x > 0$, we define the
{\em random prime running function} 
$\Tilde{\Phi}_Q(x;d,a)$ with $Q$ as a sieve modulus is a random variable
\begin{equation}
\label{eq:random-gap}
\Tilde{\Phi}_{Q}(x;d,a) := \sum_{\substack{1 \leq n \leq x\\ n \equiv a \Mod d }} W_{n,Q}.
\end{equation} 
}

\end{definition}

This definition (of a sample sequence) \eqref{eq:random-gap} parallels the definition of prime running function in ~\eqref{eq:gap_sum} in that they both sum over prime gaps (resp. $\mc{Q}$ prime gaps) with smaller prime restricted to an arithmetic progression. 
 
 The function $\Tilde{\Phi}_Q(x; d, a)$
 is of interest when $d$ divides $Q$ and $\gcd(a, d) =1$.


\subsection{Modified \cramer{} Model: Expected Value of the Random Prime Running Function}\label{sec:cramer_expected_val}

We demonstrate that the modified \cramer{} model, on average, predicts that prime running functions have a bias of the order $\frac{x}{\log x}$. 

In what follows, 
\begin{equation}
\gapcut{n}{Q} \equiv n \Mod Q,\, 1 \leq \gapcut{n}{Q} \leq Q.
\end{equation}
So $\gapcut{n}{Q}$ is least positive residue $\Mod Q$

\begin{theorem}\label{thm:MC-main}
Fix an integer $d \geq 2$ and integer $a$ such that $(a,d) = 1$. For the modified \cramer{} model with 
a fixed sieve modulus $Q$ divisible by $d$, one has 
$$
\EE[\Tilde{\Phi}_Q(x;d,a)] = \frac{x}{\totient(d)} + R_{Q}(d;a)\frac{x}{\log x} + \bigo{\frac{x}{(\log x)^2}} \quad \text{ as } x \to \infty
$$
The bias constant $R_Q(d; a)$ is given by 
\begin{equation}\label{eq:def_bias_constant}
R_Q(d;a) = R^*_Q(d;a) - \bar{R}_Q(d),
\end{equation}
where
\begin{equation}
R^*_Q(d;a) := \frac{1}{\totient(Q)^2}\sum_{\substack{s,t = 1\\(st,Q) = 1 \\s \equiv a \Mod d}}^Q \gapcut{t - s}{Q},
\end{equation}
and
\begin{equation}
  \bar{R}_Q(d) = \frac{1}{\totient(d)}\frac{Q}{\totient(Q)}\frac{\totient(Q) + 1}{2}
\end{equation}
\end{theorem}

\begin{proof}
First, we recall the definition of the prime running function for a sample of the Modified \cramer{} model,
as a function of its random variables $Z_{i,Q}$. It is 
\begin{equation}\label{eq:random_gap_sum}
  \Tilde{\Phi}_Q(x;d,a) = \sum_{\substack{n \leq x\\ n \equiv a \Mod d}} W_{n,Q}
\end{equation}
By linearity of expected values, it is sufficient to analyze the behavior of expected value of the conditional gaps $W_{n,Q}$ (definition~\ref{def:cond_gap}). By definition of expected value over a discrete space,
\begin{equation}\label{eq:expected_def}
  \EE[W_{n,Q}] = \sum_{v}v\Pr[W_{n,Q} = v].
\end{equation}

The values $v$ in \eqref{eq:expected_def} range over values of $W_{n,Q}$, which are the differences between two consecutive $\mc{Q}$-primes. 
Since only positive integers co-prime to $Q$ have a non-zero probability of being $\mc{Q}$-prime, it is helpful to introduce a notation for the unsieved integers. 
Let $U_Q$ be the set of the unsieved positive integers. i.e.
$$U_Q = \{1 = u_1 < u_2 < u_3,\ldots\} := \{u \in \cN \ |\ \gcd(u,Q) = 1\}.$$
Since the random variables $\{Z_{k,Q}\}_{k=1}^\infty$ are independent, for $u_{i+l} > u_i$ we recover that
\begin{equation}\label{eq:gap_prob}
  \Pr[W_{u_i,Q} = u_{i + l} - u_i] = \frac{\prefactor}{\log u_i}\frac{\prefactor}{\log u_{i+l}} \prod_{0 < j < l}\left(1 - \frac{\prefactor}{\log u_{i + j}}\right),
\end{equation}
where $\prefactor = \frac{Q}{\totient(Q)}$ as defined in \eqref{eq:def_prefactor}.

By substituting \eqref{eq:gap_prob} into right hand side of \eqref{eq:expected_def}, we conclude that
\begin{equation}\label{eq:ex_val}
  \EE[W_{u_i,Q}] = \frac{\prefactor^2}{\log u_i}\sum_{\substack{l > 0}} \left[\frac{u_{i+l} - u_i}{\log u_{i+l}} \prod_{0 < j < l}\left(1 - \frac{\prefactor}{\log u_{i+j}}\right)\right].
\end{equation}
While \eqref{eq:ex_val} gives us the exact value, it is difficult to work with. We proceed to approximating the expected value of $W_{u_i, Q}$ to a more convenient form.
\begin{lemma}\label{lemma:geo_simplify}
\GeoApproxLemma
\end{lemma}
The proof of Lemma \ref{lemma:geo_simplify} is postponed to the appendix.

By substituting $c = \prefactor$ and $m = 1$ into Lemma~\ref{lemma:geo_simplify}, we obtain
\begin{equation*}
\EE[W_{u_i,Q}] = \frac{\prefactor^2}{\log u_i}\sum_{l > 0}\left[\frac{u_{i + l} - u_i}{\log u_i}\left(1 - \frac{\prefactor}{\log u_i}\right)^{l-1}\right] + \bigo{\frac{(\log u_i)^\varepsilon}{u_i}}
\end{equation*} for any fixed $\varepsilon > 0$
as $u_i$ tends to infinity.

To further simplify Lemma~\ref{lemma:geo_simplify}, we separate $u_{i+l}$ into individual residue classes$\Mod Q$.
\begin{equation*}
  \EE[W_{u_i,Q}] = \frac{\prefactor^2}{(\log u_i)^2}\sum_{h = 1}^{\totient(Q)}\sum_{l \geq 0}(u_{i + \totient(Q)l + h} - u_i)\left(1 - \frac{\prefactor}{\log u_i}\right)^{\totient(Q)l + h - 1} + \bigo{\frac{(\log u_i)^\varepsilon}{u_i}}
\end{equation*}
Now let $\alpha_i = 1 - \frac{\prefactor}{\log u_i}$ and obtain
\begin{equation}\label{eq:expansion}
  \EE[W_{u_i,Q}] = \frac{\prefactor^2}{(\log u_i)^2}\sum_{h = 1}^{\totient(Q)}\alpha_i^{h - 1}\sum_{l \geq 0}(u_{i + h} - u_i + Q l)\alpha_i^{\totient(Q)l} + \bigo{\frac{(\log u_i)^\varepsilon}{u_i}}
.\end{equation}
We utilize moments of a geometric distributed random variable $Y_p$ with parameter $p \in (0,1]$.
\begin{equation}\label{eq:zero_moment}
    \EE[Y_p^0] = \sum_{h = 1}^\infty p (1 - p)^{h - 1} = 1  
\end{equation}
\begin{equation}\label{eq:first_moment}
    \EE[Y_p] = \sum_{h = 1}^\infty h p (1 - p)^{h - 1} = \frac{1}{p}  
\end{equation}
\begin{equation}\label{eq:second_moment}
    \EE[Y_p^2] = \sum_{h = 1}^\infty h^2 p (1 - p)^{h - 1} = \frac{2 - p}{p^2}
\end{equation}
More specifically consider $Y_{1 - \alpha_i^{\totient(Q)}}$.
Substituting the definition of moments to \eqref{eq:expansion}, we obtain
\begin{equation}\label{eq:expansion_geo}
\EE[W_{u_i,Q}] = \frac{\prefactor^2}{(\log u_i)^2}\sum_{h = 1}^{\totient(Q)}\paren{\alpha_i^{h - 1} \frac{u_{i + h} - u_i}{1 - \alpha_i^{\totient(Q)}}\EE\left[Y_{1 - \alpha_i^{\totient(Q)}}^0\right] + \frac{Q \alpha_i^{\totient(Q)}}{1 - \alpha_i^{\totient(Q)}}\EE\left[Y_{1 - \alpha_i^{\totient(Q)}}\right]} + \bigo{\frac{(\log u_i)^\varepsilon}{u_i}}
.    
\end{equation}

By substituting \eqref{eq:zero_moment} and \eqref{eq:first_moment} into right hand side of \eqref{eq:expansion_geo}, we obtain that
\begin{equation}\label{eq:pre_expand}
  \EE[W_{u_i,Q}] = \frac{\prefactor^2}{(\log u_i)^2}\sum_{h = 1}^{\totient(Q)}\alpha_i^{h - 1}\left[\frac{u_{i +h} - u_i}{1 - \alpha_i^{\totient(Q)}} + \frac{Q\alpha_i^{\totient(Q)}}{(1 - \alpha_i^{\totient(Q)})^2}\right] + \bigo{\frac{(\log u_i)^\varepsilon}{u_i}}
\end{equation}

We further simplify~\eqref{eq:pre_expand} using the following series expansions.
\begin{equation}\label{eq:top_expansion}
  \alpha_i^{k} = \paren{\frac{\log u_i - \prefactor}{\log u_i}}^k = 1 - \frac{k \prefactor}{\log u_i} + \frac{k(k - 1)\prefactor^2}{2 (\log u_i)^2} + \bigo{(\log u_i)^{-3}}
\end{equation}
\begin{equation}\label{eq:mid_expansion}
  \frac{1}{1 - \alpha_i^{\totient(Q)}} = \frac{\log u_i}{Q}\left(1 - \frac{(\totient(Q) - 1)\prefactor}{2 \log u_i} +  \bigo{\frac{1}{(\log u_i)^2}}\right)^{-1} = \frac{\log u_i}{Q} + \frac{\totient(Q) - 1}{2 \totient(Q)} + \bigo{(\log u_i)^{-1}}
\end{equation}
\begin{equation}\label{eq:bottom_expansion}
  \frac{\alpha_i^{\totient(Q)}}{\left(1 - \alpha_i^{\totient(Q)}\right)^2} = \frac{1}{\paren{1 - \alpha_i^{\totient(Q)}}^2} - \frac{1}{1 - \alpha_i^{\totient(Q)}} =  \frac{(\log u_i)^2}{Q^2} -\frac{1}{Q\totient(Q)}\log u_i + \bigo{1}
\end{equation}

By substituting the series expansions~\eqref{eq:top_expansion}, \eqref{eq:mid_expansion}, and \eqref{eq:bottom_expansion} into the
right hand side of \eqref{eq:pre_expand}, we obtain the following equation\footnote{For fixed $Q$, we only sum over finite number of terms in \eqref{eq:pre_expand}. Thus the constants for the Big-O type bounds are bounded.}.

\begin{eqnarray}\label{eq:sub_expansion}
  \EE[W_{u_i,Q}] = \frac{\prefactor^2}{(\log u_i)^2}\sum_{h = 1}^{\totient(Q)}\Bigg[\paren{1 - \frac{(h-1)\prefactor}{\log u_i} + \bigo{\frac{1}{(\log u_i)^2}}}\nonumber\\
  \times \paren{\frac{(\log u_i)^2}{Q} +\left(\frac{u_{i+l} - u_i}{Q} - \frac{1}{\totient(Q)}\right)\log u_i + \bigo{1}}\Bigg] + \bigo{\frac{(\log u_i)^\varepsilon}{u_i}}
\end{eqnarray}
\eqref{eq:sub_expansion} simplifies to the following.
\begin{equation}\label{eq:lambda_complexity}
  \EE[W_{u_i,Q}] = \prefactor + \frac{\prefactor^2}{\log u_i}\sum_{h = 1}^{\totient(Q)}\left(\frac{u_{i+h} - u_i}{Q} - \frac{h}{\totient(Q)}\right) + \bigo{\frac{1}{(\log u_i)^2}}
\end{equation}

Note that the projection of $\set{u_{i+h} : h = 1,2,\ldots, \totient(Q)}$ to $\paren{\cZ/Q\cZ}^\times$ is a bijection. Also note that $1 \leq u_{i+h} - u_i \leq Q$ for $1 \leq h \leq \totient(Q)$. Thus if $u_i \equiv s \Mod Q$, then
\begin{equation}
    \sum_{h = 1}^{\totient(Q)}\left(\frac{u_{i+h} - u_i}{Q} - \frac{h}{\totient(Q)}\right) = -\frac{\totient(Q) + 1}{2} + \frac{1}{Q}\sum_{\substack{1 \leq t \leq Q, (t,Q) = 1}} \gapcut{t - s}Q.
\end{equation}

Summing these contributions in \eqref{eq:lambda_complexity} yields
\begin{equation}\label{eq:partial_sum}
\sum_{\substack{u_i \leq x\\ u_i \equiv s \Mod Q}}\EE[W_{u_i,Q}] =
  \frac{x}{\totient(Q)} + \frac{Q}{\totient(Q)^2}\left(-\frac{\totient(Q) + 1}{2} + \frac{1}{Q}\sum_{\substack{1 \leq t \leq Q \\ (t, Q) = 1}} \gapcut{t - s}{Q}\right)\frac{x}{\log x} + \bigo{\frac{x}{(\log x)^2}}
\end{equation}
Finally, summing \eqref{eq:partial_sum} over all $s \equiv a \Mod d$ for $s = 1, 2, \ldots Q$ that are co-prime to $Q$, we get
\begin{equation*}
\EE[\Phi(x;q,a)] = \sum_{\substack{u_i \leq x\\ u_i \equiv a \Mod d}}\EE[W_{u_i,Q}] =
  \frac{x}{\totient(d)} + \paren{R^*_Q(d;a) - \frac{1}{\totient(d)}\frac{Q}{\totient(Q)}\frac{\totient(Q) + 1}{2}}\frac{x}{\log x} + \bigo{\frac{x}{(\log x)^2}}
\end{equation*}
\end{proof}

\subsection{Modified \cramer{} Model: Variance of the Random Prime Running Function}
The next theorem shows that the probability distribution is centered around the mean value with a standard deviation of scale at most $\sqrt{x \log x}$. Note that the standard deviation is significantly smaller than the order of bias $\frac{x}{\log x}$.

\begin{theorem}\label{thm:variance}
Fix an integer $d \geq 2$ and an integer $a$ such that $(a,d) = 1$. Then
\begin{equation}
    \var(\Tilde{\Phi}_Q(x;d,a)) = \bigo{x \log x}
.\end{equation}
\end{theorem}
\begin{proof}
As with Theorem~\ref{thm:MC-main}, let $U_Q$ be the set of the unsieved integers and let $\prefactor$ be the prefactor. i.e.
$$U_Q = \{1 = u_1 < u_2 < u_3,\ldots\} := \{u \in \cN \ |\ \gcd(u,Q) = 1\}$$
and $\prefactor = \frac{Q}{\phi(Q)}$.

We first utilize the variance of sum of random variables formula.
\begin{equation*}
    \var(\Tilde{\Phi}_Q(x;q,a)) = \sum_{\substack{u_i \leq x \\ u_i \equiv a \Mod d}} \var(W_{u_i,Q}) + 2\sum_{\substack{u_i < u_j \leq x\\ u_i \equiv u_j \equiv a \Mod d}}\cov(W_{u_i,Q}, W_{u_j,Q})
\end{equation*}
We will bound the variance and the co-variance of $W_{u_i,Q}$ separately.

By definition of variance,
\begin{equation*}
    \var(W_{u_i,Q}) = \EE[W_{u_i, Q}^2] - \EE[W_{u_i, Q}]^2 \leq \EE[W_{u_i,Q}^2].
\end{equation*}
It follows from \eqref{eq:gap_prob} that
\begin{equation}\label{eq:var_bound}
    \EE[W_{u_i,Q}^2] = \frac{\prefactor^2}{\log u_i} \sum_{l > 0}\left[\frac{\paren{u_{i + l} - u_i}^2}{\log u_{i + l}}\prod_{0 < j < l}\paren{1 - \frac{\prefactor}{\log u_{i + j}}}\right]
.\end{equation}
By Lemma~\ref{lemma:geo_simplify} and the inequality $u_{i + l} - u_i \leq Q l$, \eqref{eq:var_bound} simplifies to
\begin{equation}\label{eq:var_bound_simplify}
    \var(W_{u_i,Q}) \leq \frac{\prefactor^2Q^2}{\log u_i} \sum_{l > 0}\left[\frac{l^2}{\log u_{i}}\paren{1 - \frac{\prefactor}{\log u_i }}^{l - 1}\right] + \bigo{\frac{(\log u_i)^4}{u_i}}
.\end{equation}
Letting $Y_p$ be a geometrically distributed random variable with parameter $p = \frac{\prefactor}{\log u_i}$. By substituting equation for $\EE[Y_p^2]$ into \eqref{eq:var_bound_simplify}, we obtain that
\begin{equation}
    \var(W_{u_i,Q}) \leq \frac{\prefactor Q^2}{\log u_i} \EE[Y_p^2] + \bigo{\frac{(\log u_i)^4}{u_i}}
.\end{equation}

By second moment of geometric distribution \eqref{eq:second_moment}, we obtain that
\begin{equation}
    \var(W_{u_i,Q}) = \bigo{\log u_i}
.\end{equation}
Thus there exists a constant $C > 0$ such that $\var(W_{u_i,Q}) \leq C \log (u_i)$ for sufficiently large $u_i$. We obtain that
\begin{equation}
    \sum_{\substack{u_i \leq x\\ u_i \equiv a \Mod d}}\var(W_{u_i,Q}) \leq C x\log x + \bigo{1}
\end{equation}

We now bound the covariance terms. Suppose that $i < j$. We will split the covariance into parts by conditioning on different events.
\begin{equation}
    \cov(W_{u_i,Q}, W_{u_j,Q}) = \sum_{k = 1}^3
    \EE[W_{u_i, Q}W_{u_j, Q} | E_k]\PP(E_k) - \EE[W_{u_i, Q}]\EE[W_{u_j, Q}]
\end{equation}
where $E_1, E_2, E_3$ are events $W_{u_i, Q} = u_j - u_i$, $W_{u_i, Q} < u_j - u_i$, $W_{u_i, Q} > u_j - u_i$ respectively.
Suppose $W_{u_i, Q} > u_j - u_i$ ($E_3$). Then $u_j$ cannot be $\mc{Q}$- prime. Such event implies that $W_{u_j, Q} = 0$. Thus $\EE[W_{u_i, Q}W_{u_j, Q} | E_3] = 0$.
Now suppose that $W_{u_i, Q} < u_j - u_i$ ($E_2$). Such event implies that $W_{u_i,Q}$ and $W_{u_j,Q}$ are (conditionally) independent. Thus $\EE[W_{u_i, Q}W_{u_j, Q} | E_2] = \EE[W_{u_i, Q}|E_2]\EE[W_{u_j, Q}] \le \EE[W_{u_i, Q}]\EE[W_{u_j,Q}]$.
By combining these two observations, we conclude that
\begin{equation*}
    \cov(W_{u_i,Q}, W_{u_j,Q}) \leq \PP(W_{u_i,Q} = u_j - u_i)\EE(W_{u_i,Q} W_{u_j,Q} | W_{u_i,Q} = u_j - u_i)
,\end{equation*}
which simplifies to
\begin{equation*}
    \cov(W_{u_i,Q}, W_{u_j,Q}) \leq \PP(W_{u_i,Q} = u_j - u_i)](u_j - u_i)\frac{\log u_j}{\prefactor} \EE[W_{u_j,Q}]
.\end{equation*}

By \eqref{eq:lambda_complexity}, $\EE[W_{u_j,Q}] = \bigo{1}$ and by \eqref{eq:gap_prob}, $\PP(W_{u_i,Q} = u_j - u_i) = \bigo{\frac{1}{(\log u_i)^2}\paren{1 - \frac{\prefactor}{\log u_j} }^{j-i}}$. Thus there exists a constant $A > 0$ such that for sufficiently large $u_i$ and $u_j$,
\begin{equation}\label{eq:cov_basic_bound}
    \cov(W_{u_i,Q}, W_{u_j,Q}) \leq A\frac{\log u_j}{(1 + \log u_i)^2}(u_j - u_i)\paren{1 - \frac{\prefactor}{\log u_j}}^{j - i}
\end{equation}
Note that instead of $(\log u_i)^2$ as the denominator in equation \ref{eq:cov_basic_bound}, we have $(1 + \log u_i)^2$. This allows us to avoid dividing by $0$ when $u_i = 1$. This inconvenience occurs because the probability that $u_i$ is $\mc{Q}$ prime is equal to $\frac{\prefactor}{\log u_i}$ only if $u_i$ is sufficiently large. By summing equation \ref{eq:cov_basic_bound} over different values of $u_i, u_j$, we obtain
\begin{equation}\label{eq:cov_sum}
    \sum_{\substack{u_i < u_j \leq x\\ u_i \equiv u_j \equiv a \Mod d}}\cov(W_{u_i,Q}, W_{u_j,Q}) \leq A \sum_{\substack{u_i < u_j \leq x \\ u_i \equiv u_j \equiv a \Mod d}} \frac{\log u_j}{(1 + \log u_i)^2} (u_j - u_i) \paren{1 - \frac{\prefactor}{\log u_j}}^{j - i} + \bigo{1}
.\end{equation}
By utilizing \eqref{eq:unseive_gap_bound}, $\log u_j \leq \log x$, $u_j - u_i \leq Q(j - i)$, and adding additional non-negative terms, \eqref{eq:cov_sum} simplifies to
\begin{equation}\label{eq:cov_sum_simp}
    \sum_{\substack{u_i < u_j \leq x\\ u_i \equiv u_j \equiv a \Mod d}}\cov(W_{u_i,Q}, W_{u_j,Q}) \leq A Q \sum_{u_i \leq x} \frac{\log x}{(1 + \log u_i)^2} \sum_{h = 1}^\infty h \paren{1 - \frac{\prefactor}{\log x}}^h + \bigo{1}
.\end{equation}
By setting $Y_{\frac{\prefactor}{\log x}}$ to be a geometric random variable with parameter $p = \frac{\prefactor}{\log x}$ and substituting the definition of $\EE[Y_{\frac{\prefactor}{\log x}}]$ into \eqref{eq:cov_sum_simp}, we obtain that
\begin{equation}
    \sum_{\substack{u_i < u_j \leq x\\ u_i \equiv u_j \equiv a \Mod d}}\cov(W_{u_i,Q}, W_{u_j,Q}) \leq A Q \sum_{u_i \leq x} \frac{\log x}{(1 + \log u_i)^2} \frac{\log x}{\prefactor} \EE\left[Y_{\frac{\prefactor}{\log x}}\right] + \bigo{1}
.\end{equation}
By \eqref{eq:first_moment} and the fact $\sum_{n \leq x} \frac{1}{(1 + \log n)^2} = \bigo{\frac{x}{(\log x)^2}}$, we obtain
\begin{equation}
    \sum_{\substack{u_i < u_j \leq x \\ u_i \equiv u_j \equiv a \Mod d}}\cov(W_{u_i,Q}, W_{u_j,Q}) \leq \bigo{x \log x}
.\end{equation}

\end{proof}

\subsection{Modified \cramer{} Model: Anti-symmetry Properties}
Theorem~\ref{thm:anti_symmetry} suggests that the bias constant anti-symmetry Conjecture (\ref{conj:odd_running_error}) should be true.

\begin{theorem}\label{thm:anti_symmetry}
For any integer $d \geq 2$ and integer $Q$ divisible by $d$, the following anti-symmetry holds.
$$R_Q(d;-a) = - R_Q(d;a)$$
\end{theorem}

\begin{proof}
Fix an integer $a$ co-prime to $Q$.
By definition of bias constants,
\begin{equation*}
    R_Q^*(d;-a) = \frac{1}{\totient(Q)^2}\sum_{\substack{s,t = 1\\ s \equiv -a \Mod d\\ (st, Q) = 1}}^Q\gapcut{t - s}{Q}
.\end{equation*}
Because $t \to -t$ is a permutation of $\paren{\cZ/Q\cZ}^\times$, we can sum over $-t$. Furthermore, $s \equiv a \Mod d$ implies $-s \equiv -a \Mod d$. Thus
\begin{equation*}
    R_Q^*(d;-a) = \frac{1}{\totient(Q)^2}\sum_{\substack{s,t = 1\\ s \equiv a \Mod d\\ (st, Q) = 1}}^Q\gapcut{s - t}{Q}
.\end{equation*}
It follows that
\begin{equation}\label{eq:anti_sum}
    R_Q^*(d;a) + R_Q^*(d;-a) = \frac{1}{\totient(Q)^2}\sum_{\substack{s,t = 1\\ s \equiv a \Mod d\\ (st, Q) = 1}}^Q\paren{\gapcut{t - s}{Q} + \gapcut{s - t}{Q}}
.\end{equation}
Note that $\gapcut{t - s}{Q} + \gapcut{s - t}{Q} \equiv 0 \Mod Q$ and $\gapcut{t - s}{Q} + \gapcut{s - t}{Q} \in [2, 2Q]$. Thus
\begin{equation}\label{eq:anti_sum_summand}
    \gapcut{t - s}{Q} + \gapcut{s - t}{Q} = 
    \begin{cases}
        Q & t \not\equiv s \Mod Q\\
        2Q & t \equiv s \Mod Q\\
    \end{cases}
.\end{equation}
By counting the number of times $s = t$ in \eqref{eq:anti_sum} and utilizing \eqref{eq:anti_sum_summand}, we obtain that
\begin{equation}
    R_Q^*(d;a) + R_Q^*(d;-a) = \frac{1}{\totient(Q)^2}\paren{Q\frac{\totient(Q)^2}{\totient(d)} + Q\frac{\totient(Q)}{\totient(d)}} = 2\bar{R}_Q(d)
.\end{equation}
\end{proof}
As an immediate corollary, we obtain that the bias constants add up to $0$.
\begin{corollary}\label{cor:zero_sum}
For $d \geq 2$ and $Q$ divisible by $d$,
\begin{equation}
    \sum_{\substack{a = 1 \\ (a, d) = 1}}^d R_Q(d;a) = 0
\end{equation}
and
\begin{equation}
    \bar{R}_Q(d) = \frac{1}{\totient(d)}\sum_{\substack{a = 1 \\ (a,d) = 1}}^d R_Q^*(d;a)
\end{equation}
\end{corollary}

\subsection{Modified \cramer{} Model: Radical Equivalence Property}
\begin{theorem}\label{thm:radical}
    For all $d \geq 2$ with $(a,d) = 1$ and $d | Q$,
    \begin{equation}\label{eq:radical_dependence}
        R_Q(d;a) = \frac{\totient(d_{sf})}{\totient(d)}R_Q(d_{sf}, a),
    \end{equation}
    where $d_{sf} = rad(d)$ is the maximal square-free divisor of $d$.
    Equivalently,
    \begin{equation}\label{eq:radical_equiv}
        R_Q(d;a) = R_Q(d; a')
    \end{equation}
    if $a \equiv a' \Mod{d_{sf}}$.
\end{theorem}
\begin{proof}
Note that for any fixed sample sequence of $\mc{Q}$ primes,
\begin{equation}
    \Tilde{\Phi}_Q(d_{sf};a) = \sum_{\substack{a' = 1\\ a' \equiv a \Mod{d_{sf}}}}^d \Tilde{\Phi}_Q(d;a').
\end{equation}
By linearity of expected value and inspecting the $\frac{x}{\log x}$ order term from Theorem \ref{thm:MC-main}, we obtain that
\begin{equation}
    R_Q(d_{sf};a) = \sum_{\substack{a' = 1\\ a' \equiv a \Mod{d_{sf}}}}^d R_Q(d;a')
.\end{equation}
Thus \eqref{eq:radical_equiv} implies \eqref{eq:radical_dependence}. Fix $a$ and $a'$ such that $a \equiv a' \Mod{d_{sf}}$.
Let $Q_{sf} = rad(Q)$ denote the square-free part of $Q$. Because $d_{sf}$ divides $Q_{sf}$ and $Q_{sf} / d_{sf}$ is co-prime to $d$, there exists some integer $k$ such that $a + kQ_{sf} = a' \Mod d$. Thus it suffices to show that $R_Q(d;a) = R_Q(d;a + Q_{sf})$ for any fixed $a$. By definition of bias constants given in Theorem~\ref{thm:MC-main},
\begin{equation*}
    R_Q(d;a + Q_{sf}) = -\bar{R}_Q(d) +  \frac{1}{\totient(Q)^2}\sum_{\substack{s,t = 1\\ (st, Q) = 1\\ s \equiv a + Q_{sf}\Mod d}}^Q\gapcut{t - s}{Q}
\end{equation*}
Since $\gapcut{t - s}{Q}$ only depends on value of $t - s \Mod Q$, we can sum over $s \equiv a \Mod d$ and then add $Q_{sf}$ to $s$.
\begin{equation*}
    R_Q(d;a + Q_{sf}) = -\bar{R}_Q(d) +  \frac{1}{\totient(Q)^2}\sum_{\substack{s,t = 1\\ (st, Q) = 1\\ s \equiv a \Mod d}}^Q\gapcut{t - (s + Q_{sf})}{Q}
\end{equation*}
\begin{equation*}
    R_Q(d;a + Q_{sf}) = -\bar{R}_Q(d) +  \frac{1}{\totient(Q)^2}\sum_{\substack{s,t = 1\\ (st, Q) = 1\\ s \equiv a \Mod d}}^Q\gapcut{(t - Q_{sf}) - s}{Q}
\end{equation*}
\begin{equation*}
    R_Q(d;a + Q_{sf}) = -\bar{R}_Q(d) +  \frac{1}{\totient(Q)^2}\sum_{\substack{s,t = 1\\ (s(t + Q_{sf}), Q) = 1\\ s \equiv a \Mod d}}^Q\gapcut{t - s}{Q}
\end{equation*}
Well, for any integer $t$, and any prime factor $p$ of $Q$, $t \equiv t + Q_{sf} \Mod p$. Thus $t$ is co-prime to $Q$ if and only if $t + Q_{sf}$ is co-prime to $Q$. It follows that
\begin{equation}\label{eq:square_free_part_final_eq}
    R_Q(d;a + Q_{sf}) = -\bar{R}_Q(d) +  \frac{1}{\totient(Q)^2}\sum_{\substack{s,t = 1\\ (st, Q) = 1\\ s \equiv a \Mod d}}^Q\gapcut{t - s}{Q}.
\end{equation}
We are done because \eqref{eq:square_free_part_final_eq} is the definition of $R_Q(d;a)$.
\end{proof}

\section{Computation For Modified \cramer{} Model} \label{sec:bias_compute}
In this section, we compute the bias constants $R_Q(d;a)$ for the modified \cramer{} model for various values of $Q$ and $d$.

\subsection{Recursive Formula for Bias Constants}
Brute force computation of bias constant $R_Q(d;a)$ has runtime complexity that is polynomial in $Q$, which  is
 exponential  in input bit size $O( \log Q)$. 
The following result gives a recursive formula yielding an improved method for computing the bias constants $R_Q(d;a)$ for fixed $d$ and all $a \Mod d$ with $(a,d) = 1$.
\begin{theorem}\label{thm:recursion}
Suppose $d, p, Q_0 \geq 2$ are pairwise co-prime and $p$ is a prime. Let $Q = dQ_0$. Then
\begin{equation}\label{eq:recursion_map}
R_{pQ}(d;pa) = \frac{\totient(p)^2 - 1}{\totient(p)^2}R_{Q}(d;pa) + \frac{p}{\totient(p)^2}R_Q(d;a)
\end{equation}
\end{theorem}
\begin{definition}
    Given $Q_1,\ldots, Q_k$ pairwise co-prime, we define $\gapcut{n_1,\ldots, n_k}{Q_1,\ldots, Q_k}$ to be the unique element in $[1, Q_1Q_2\ldots Q_k]$ such that $$\gapcut{n_1,\ldots, n_k}{Q_1,\ldots, Q_k} \equiv n_i \Mod{Q_i} \quad i = 1, \ldots, k$$
\end{definition}
Note that the definition is consistent with the definition of least positive residue $\gapcut{n}{Q}$. Because $\gapcut{n}{Q_1Q_2\ldots Q_k}$ is congruent to $n \Mod{Q_i}$ for $i = 1,\ldots, k$, we obtain
\begin{equation}\label{eq:chinese_decomp}
  \gapcut{n}{Q_1Q_2\ldots Q_k} = \gapcut{n,\ldots, n}{Q_1,\ldots, Q_k}
.\end{equation}

\begin{proof}[Proof of Theorem \ref{thm:recursion}]
By corollary \ref{cor:zero_sum}, 
\begin{equation}\label{eq:bias_const_sum}
R_{pQ}(d;pa) := \frac{1}{\totient(pQ)^2}\sum_{\substack{s,t = 1\\(st,pQ) = 1 \\s \equiv a \Mod d}}^{pQ} \gapcut{t - s}{pQ} - \frac{1}{\totient(pQ)^2\totient(d)}\sum_{\substack{s,t = 1\\(st,pQ) = 1}}^{pQ} \gapcut{t - s}{pQ}.
\end{equation}
By substituting \eqref{eq:chinese_decomp} into \eqref{eq:bias_const_sum}, we obtain that
\begin{equation}\label{eq:R_def}
  R_{pQ}(d;pa) = \frac{1}{\totient(pQ)^2}\left(\sum_{\substack{s = 1\\s \equiv pa \Mod d\\ (s,pQ) = 1}}^{pQ} \sum_{\substack{t = 1\\ (t,pQ) = 1}}^{pQ} \gapcut{t - s, t - s}{pQ_0,d} - \frac{1}{\totient(d)}\sum_{\substack{s,t = 1\\ (st,pQ) = 1}}^{pQ} \gapcut{t - s, t - s}{pQ_0,d}\right).
\end{equation}
The restriction of $s \equiv pa \Mod d$ ensures that the summand $\gapcut{t - s, t - s}{pQ_0, d}$ is congruent to $pa \Mod d$. Fortunately, we can eliminate the restriction by directly forcing the summand to be equivalent to $pa \Mod d$ by the following identity.

\begin{prop}
Suppose $d, p, Q_0 \geq 2$ are pairwise co-prime and $p$ is a prime. Let $Q = dQ_0$. Then for $a$ co-prime to $d$,
\begin{equation}\label{eq:sum_other_res_formula}
  \frac{1}{\totient(d)}\sum_{\substack{s,t = 1\\ (st,pQ) = 1}}^{pQ} \gapcut{t - s, t - pa}{pQ_0, d} = \sum_{\substack{s = 1\\s \equiv pa \Mod d\\ (s,pQ) = 1}}^{pQ} \sum_{\substack{t = 1\\ (t,pQ) = 1}}^{pQ} \gapcut{t - s, t - s}{pQ_0,d} 
.\end{equation}
\end{prop}
\begin{proof}
Because $pQ_0$ and $d$ are co-prime, $\paren{\cZ/pQ\cZ}^\times \cong \paren{\cZ/pQ_0\cZ}^\times \times \paren{\cZ/d\cZ}^\times$. Thus
\begin{equation}\label{eq:double_chinese}
    \frac{1}{\totient(d)}\sum_{\substack{s_1 = 1\\ (s_1, pQ_0) = 1}}^{pQ_0}\sum_{\substack{s_2 = 1\\ (s_2, d) = 1}}^{d} \sum_{\substack{t = 1\\ (t,pQ) = 1}}^{pQ} \gapcut{t - \gapcut{s_1, s_2}{pQ_0, d}, t - pa}{pQ_0, d} =
    \frac{1}{\totient(d)}\sum_{\substack{s,t = 1\\ (st,pQ) = 1}}^{pQ} \gapcut{t - s, t - pa}{pQ_0, d}.
\end{equation}
Note that $t - \gapcut{s_1, s_2}{pQ_0, d} \equiv t - s_1 \Mod{pQ_0}$. Thus the summand on the RHS of \eqref{eq:double_chinese} is independent of $s_2$. By replacing all $s_2$ with $pa$, we obtain that
\begin{equation*}
    \frac{1}{\totient(d)}\sum_{\substack{s,t = 1\\ (st,pQ) = 1}}^{pQ} \gapcut{t - s, t - pa}{pQ_0, d} = \sum_{\substack{s_1 = 1\\ (s_1, pQ_0) = 1}}^{pQ_0} \sum_{\substack{t = 1\\ (t,pQ) = 1}}^{pQ} \gapcut{t - \gapcut{s_1, pa}{pQ_0, d}, t - pa}{pQ_0, d}
.\end{equation*}
Note that
$$\set{\gapcut{s_1, pa}{pQ_0, d} : 1 \leq s_1 \leq pQ_0, (s_1, pQ_0) = 1} = \set{s : 1 \leq s \leq pQ, (s,pQ) = 1, s \equiv pa \Mod d}.$$ Thus
\begin{equation}\label{eq:recursion:sum_over_all}
    \frac{1}{\totient(d)}\sum_{\substack{s,t = 1\\ (st,pQ) = 1}}^{pQ} \gapcut{t - s, t - pa}{pQ_0, d} = \sum_{\substack{s = 1\\ s \equiv pa \Mod d\\ (s, pQ) = 1}}^{pQ} \sum_{\substack{t = 1\\ (t,pQ) = 1}}^{pQ} \gapcut{t - s, t - pa}{pQ_0, d}.\end{equation}
For $s \equiv pa \Mod d$, $t - pa \equiv t - s \Mod d$. Thus we can replace $t - pa$ with $t - s$ in the summand of the RHS of \eqref{eq:recursion:sum_over_all}.
\end{proof}
By substituting \eqref{eq:sum_other_res_formula} into \eqref{eq:R_def},
\begin{equation}\label{eq:sum_other_res}
  R_{pQ}(d;pa) = \frac{1}{\totient(pQ)^2\totient(d)}\sum_{\substack{s,t = 1\\ (st,pQ) = 1}}^{pQ} \Big(\gapcut{t - s, t - pa}{pQ_0, d} - \gapcut{t - s, t - s}{pQ_0, d}\Big)
\end{equation} and similarly,
\begin{equation}\label{eq:sum_other_res_Q}
  R_{Q}(d;a) = \frac{1}{\totient(Q)^2\totient(d)}\sum_{\substack{s,t = 1\\ (st,Q) = 1}}^{Q} \Big(\gapcut{t - s, t - a}{Q_0, d} - \gapcut{t - s, t - s}{Q_0, d}\Big)
\end{equation}

We now decompose \eqref{eq:sum_other_res} by the decomposition $\paren{\cZ/pQ_0\cZ}^\times \cong \paren{\cZ/p\cZ}^\times \times \paren{\cZ/Q_0\cZ}^\times$.
\begin{equation}\label{equ:factor_p}
  R_{pQ}(d;pa) = \frac{1}{\totient(pQ)^2\totient(d)}\sum_{\substack{s,t = 1\\ (st,Q) = 1}}^{Q} \sum_{s',t' = 1}^{p-1}\Big(\gapcut{t - s, t - pa, t' - s'}{Q_0, d, p} - \gapcut{t - s, t - s, t' - s'}{Q_0, d,p}\Big)
\end{equation}

Note that for any $r_1, r_2, r_3 \in \cZ$, 
$\gapcut{r_1, r_2, r_3}{Q_0, d, p} - \gapcut{r_1, r_2}{Q_0, d}$ is congruent to $0 \Mod Q$ and $r_3 - \gapcut{r_1, r_2}{Q_0, d} \Mod p$. By Chinese remainder theorem, for any fixed $r_1, r_2 \in \cZ$,
$$r_3 \mapsto \frac{1}{Q}(\gapcut{r_1, r_2, r_3}{Q_0,d,p} - \gapcut{r_1, r_2}{Q_0,d})$$
is a permutation on $\{0, 1, 2,\ldots, p- 1\}$.
By further fixing $r_2' \in \cZ$ and summing over the set $\{0, 1,\ldots, p - 1 \}$, we conclude that
\begin{equation}\label{eq:sum_permutation}
  \sum_{r_3 = 0}^{p-1}\paren{\gapcut{r_1,r_2, r_3}{Q_0,d,p} - \gapcut{r_1, r_2', r_3}{Q_0,d,p}} = p(\gapcut{r_1, r_2}{Q_0,d} - \gapcut{r_1, r_2'}{Q_0,d})
\end{equation}
We apply \eqref{eq:sum_permutation} to \eqref{equ:factor_p} as we sum over $s'$.
\begin{align*}
  R_{pQ}(d;pa) &= \frac{1}{\totient(pQ)^2\totient(d)}\sum_{\substack{s,t = 1\\ (st,Q) = 1}}^{Q} \sum_{t' = 1}^{p-1} p(\gapcut{t - s, t - pa}{Q_0, d} - \gapcut{t - s, t - s}{Q_0, d}) - \\
  & \frac{1}{\totient(pQ)^2\totient(d)}\sum_{\substack{s,t = 1\\ (st,Q) = 1}}^{Q} \sum_{t' = 1}^{p-1} (\gapcut{t - s, t - pa, t'}{Q_0, d, p} - \gapcut{t - s, t - s, t'}{Q_0, d, p})
\end{align*}
We apply \eqref{eq:sum_permutation} once more by summing over $t'$.
\begin{align*}
  R_{pQ}(d;pa) &= \frac{p(p-1)}{\totient(pQ)^2\totient(d)}\sum_{\substack{s,t = 1\\ (st,Q) = 1}}^{Q} (\gapcut{t - s, t - pa}{Q_0, d} - \gapcut{t - s, t - s}{Q_0, d}) -\\
  & \frac{1}{\totient(pQ)^2\totient(d)}\sum_{\substack{s,t = 1\\ (st,Q) = 1}}^{Q} p(\gapcut{t - s, t - pa}{Q_0, d} - \gapcut{t - s, t - s}{Q_0, d})+
  \\
  &\frac{1}{\totient(pQ)^2\totient(d)}\sum_{\substack{s,t = 1\\ (st,Q) = 1}}^{Q} (\gapcut{t - s, t - pa, 0}{Q_0,d,p} - \gapcut{t - s, t - s, 0}{Q_0, d, p})
\end{align*}
This simplifies to
\begin{align}\label{eq:sum_simplified}
  R_{pQ}(d;pa) &= \frac{p(p-2)}{\totient(p)^2}\frac{1}{\totient(Q)^2\totient(d)}\sum_{\substack{s,t = 1\\ (st,Q) = 1}}^{Q} (\gapcut{t - s, t - pa}{Q_0, d} - \gapcut{t - s, t - s}{Q_0, d})+ \nonumber
  \\
  &\frac{1}{\totient(pQ)^2\totient(d)}\sum_{\substack{s,t = 1\\ (st,Q) = 1}}^{Q} (\gapcut{t - s, t - pa, 0}{Q_0,d,p} - \gapcut{t - s, t - s, 0}{Q_0, d, p})
\end{align}
By \eqref{eq:sum_other_res_Q} the first term of \eqref{eq:sum_simplified} is $\frac{\totient(p)^2 - 1}{\totient(p)^2}R_Q(d;pa)$.
\begin{align}
  R_{pQ}(d;pa) &= \frac{\totient(p)^2 - 1}{\totient(p)^2}R_Q(d;pa)+\nonumber
  \\
  &\frac{1}{\totient(pQ)^2\totient(d)}\sum_{\substack{s,t = 1\\ (st,Q) = 1}}^{Q} (\gapcut{t - s, t - pa, 0}{Q_0,d,p} - \gapcut{t - s, t - s, 0}{Q_0, d, p})
\end{align}
Note that multiplication by $p$ is a permutation of $\paren{\cZ/ Q_0 \cZ}^\times$ and $\paren{\cZ/ d \cZ}^\times$. Thus one could sum over $ps$ and $pt$ instead of $s$ and $t$.
\begin{align}
  R_{pQ}(d;pa) &= \frac{\totient(p)^2 - 1}{\totient(p)^2}R_Q(d;pa)+\nonumber
  \\
  &\frac{1}{\totient(pQ)^2\totient(d)}\sum_{\substack{s,t = 1\\ (st,Q) = 1}}^{Q} (\gapcut{pt - ps, pt - pa, 0}{Q_0,d,p} - \gapcut{pt - ps, pt - ps, 0}{Q_0, d, p})
\end{align}
By the Chinese remainder theorem, 
\begin{equation}\label{eq:chinese_p}
  \gapcut{pr_1, pr_2, 0}{Q_0,d,p} = p\gapcut{r_1, r_2}{Q_0,d}\quad r_1, r_2 \in \cZ.
.\end{equation}
\begin{align}
  R_{pQ}(d;pa) &= \frac{\totient(p)^2 - 1}{\totient(p)^2}R_Q(d;pa)+\nonumber
  \\
  &p\frac{1}{\totient(pQ)^2\totient(d)}\sum_{\substack{s,t = 1\\ (st,Q) = 1}}^{Q} (\gapcut{t - s, t - a}{Q_0,d} - \gapcut{t - s, t - s}{Q_0, d})
\end{align}
On substituting \eqref{eq:sum_other_res_Q}, we conclude that
\begin{equation}
  R_{pQ}(d;pa) = \frac{\totient(p)^2 - 1}{\totient(p)^2}R_Q(d;pa)
  + \frac{p}{\totient(p)^2}R_Q(d;a)
\end{equation}
\end{proof}

\subsection{Computation of Modified \cramer{} Bias Constants}
We compute bias constants $R_Q(d;a)$ utilizing the recursive algorithm in Theorem~\ref{thm:recursion}.
For modulus $d = p$ a prime, the simplest case is $Q = d$, and the bias constant is given by 
\begin{equation}
    R_d(d;a) = \frac{d}{\totient(d)^2}\paren{\frac{a}{d} - \frac{1}{2}} \quad 1 \leq a \leq d - 1
.\end{equation}
These constants $R_d(d;a)$ are increasing as a function of $a$ for $1 \leq a \leq d - 1$. 
For $d = 3,5,7$ $R_d(d;a)$ significantly differ from the empirical data on bias constants $R(x;d,a)$ given in Tables~\ref{tab:r_mod3}, \ref{tab:r_mod5}, \ref{tab:r_mod7} in Section~\ref{sec:3}. 
The empirical data also disagrees in sign for $d = 3$ and the constants oscillate in $a$ for $d = 5$ and $d = 7$.

We now study the effect of larger sieve modulus $Q$ on the modified \cramer{} bias constants, which seems to improves our numerical result.
In particular, we consider the case of a modified \cramer{} model with an initial sieve over all
the prime numbers less than or equal to $T$. We let our sieve modulus $Q = T\#$, where the {\em primorial at $T$}, is defined by
\begin{equation}
\label{def: primorial}
T\# := \prod_{p \le T} p.
\end{equation}
The notation $T\#$  for primorials follows Caldwell and Gallot \cite{CaldwellG2002}. 
Thus $\Tilde{\Phi}_{T\#}(x;d,a)$ is a random prime running function corresponding to 
the modified \cramer{} model with initial sieving by all primes less than or equal to $T$.

Tables~\ref{tab:more_sieve_mod3}, \ref{tab:more_sieve_mod5}, and \ref{tab:more_sieve_mod7} give values
of \cramer{} bias constants at various primorials. 

\begin{table}[H]
\centering
	\begin{tabular}{|c|c|c|c|c|c|c|}
	\hline
	    &\multicolumn{5}{|c|}{\cramer{} model bias constants} & rescaled bias function\\
	\hline
        \diagbox{a}{Q} & $Q = 3$ & $Q = 3\#$ & $Q = 10\#$ & $Q = 100\#$ & $Q = 1000\#$ & $R(10^{12}; 3, a)$\\
    \hline
		a = 1 & -0.125 & 0.25 & 0.1823 & 0.1599 & 0.1569 & 0.2022\\
	\hline
		a = 2 & 0.125 & -0.25 & -0.1823 & -0.1599 & -0.1569 & -0.2022\\
	\hline
	\end{tabular}
	\caption{Table of bias constant $R_{Q}(3;a)$ for various sieve moduli $Q$. The right most column is the empirical data $R(10^{12};3,a)$.}
	\label{tab:more_sieve_mod3}
\end{table}

\begin{table}[H]
\centering
	\begin{tabular}{|c|c|c|c|c|c|c|}
	\hline
	    &\multicolumn{5}{|c|}{\cramer{} model bias constants} & rescaled bias function\\
	\hline
        \diagbox{a}{Q} & $Q = 5$ & $Q = 5\#$ & $Q = 10\#$ & $Q = 100\#$ & $Q = 1000\#$ & $R(10^{12}; 5, a)$\\
    \hline
		a = 1 & -0.09375 & -0.0938 & -0.0547 & -0.0699 & -0.0685 & -0.0703\\
	\hline
		a = 2 & -0.03125 & -0.1875 & -0.2005 & -0.2027 & -0.2043 & -0.221\\
	\hline
        a = 3 & 0.03125 & 0.1875 & 0.2005 & 0.2027 & 0.2043 & 0.2059\\
	\hline
        a = 4 & 0.09375 & 0.0938 & 0.0547	& 0.0699 & 0.0685 & 0.0855\\
	\hline
	\end{tabular}
	\caption{Table of bias constant $R_{Q}(5;a)$ for various sieve moduli $Q$. The right most column is the empirical data $R(10^{12};5,a)$.}
	\label{tab:more_sieve_mod5}
\end{table}

\begin{table}[H]
\centering
	\begin{tabular}{|c|c|c|c|c|c|} 
	\hline
	&\multicolumn{4}{|c|}{\cramer{} model bias constants} & rescaled bias function\\
	\hline
    \diagbox{a}{Q} & $Q = 7$ & $Q = 10\#$ & $Q = 100\#$ & $Q = 1000\#$ & $R(10^{12};7;a)$\\
    \hline
		a = 1 &-0.0964 & 0.1432 & 0.1303 & 0.1310 & 0.1461\\
		\hline
		a = 2 & -0.0417 & -0.0781 & -0.0749 & -0.0753 & -0.0680\\
		\hline
      a = 3 & -0.0139 & 0.0651 & 0.0554 & 0.0557 & 0.0506\\
		\hline
       a = 4 & 0.0139 & -0.0651	& -0.0554 & -0.0557 & -0.0571\\
		\hline
		a = 5 & 0.0417 & 0.0781	& 0.0749 & 0.0753 & 0.0626\\
		\hline
		a = 6 & 0.0964 & -0.1432	& -0.1303 & -0.1310 & -0.1343\\
		\hline
	\end{tabular}
	\caption{Table of bias constant $R_{Q}(7;a)$ for various sieve moduli $Q$. The right most column is the empirical data $R(10^{12};7,a)$.}
	\label{tab:more_sieve_mod7}
\end{table}

The bias constant for the expected values in these modified \cramer{} models with sieve modulus of $Q=1000\#$ exhibit numerical resemblance with the empirical data for $d = 5$ and $7$. However, for the case $d = 3$, there are significant deviations from the empirical data.

Note that as $T$ varies in these tables, the values of the constants $R_{T\#}(d; a)$ may be showing oscillations as $T$ increases.

\section{Concluding Remarks}\label{sec:conclusion}
Section \ref{sec:cramer}  presented a modified \cramer{} model which exhibits  a mechanism that can lead to biases of order $x / \log x$. 
Our data in Section \ref{sec:bias_compute} computes  bias constants for this model for primorials $T\#$ that roughly agree with the empirical data in Section \ref{sec:3} for $d = 5$ and $d = 7$.

The choice of taking the sieve modulus $Q$ to run through primorials $T\#$  in the modified  \cramer{} model  is significant.
Based on the choice of the sequence of integers $\set{S_i}_{i = 1}^\infty$ with $S_i | S_{i + 1}$, $R_{S_i}(d;a)$ could diverge or converge to a value that depends on the choice of $\set{S_i}_{i = 1}^\infty$.
For example, fix $d \geq 2$ prime and choose $a$ with $(a,d) = 1$. 
Define 
$$Q_T = d \prod_{\substack{p \leq T\\ p \equiv 1 \Mod d}}p.$$
By Theorem~\ref{thm:recursion},
\begin{equation}
    R_{Q_T}(d;a) = \paren{\prod_{\substack{p \leq T\\ p \equiv 1 \Mod d}} \frac{\phi(p)^2 + p - 1}{\phi(p)^2}} R_d(d;a) = \paren{\prod_{\substack{p \leq T\\ p \equiv 1 \Mod d}} \frac{p}{p - 1}} R_d(d;a).
\end{equation}
It is known that
$$\paren{\prod_{\substack{p \leq T\\ p \equiv 1 \Mod d}} \frac{p}{p - 1}} \sim c (\log(x))^{1/\phi(d)}$$
for some constant $c > 0$
(see \cite{LanMerten}, \cite{amertens}).
In particular,  for this choice of $Q_T$, the constants $R_{Q_T}(d;a)$ diverge 
as $T$ grows to infinity.

We do not address the question of whether the bias constants $R_Q(d;a)$ produced by this model (letting $Q \to \infty$ through the primorials) will necessarily agree with the bias constants $R(d;a)$ 
asserted to exist in Conjecture~\ref{conj:running_error}.

We defined the prime running functions $\Phi(x;d,a)$ as summing  gaps between primes $p_k \equiv a \Mod d$ below $x$ and the next following prime $p_{k+1}$, up to $x$. However, one 
also  consider the \emph{reversed prime running functions} $\Phi^R(x;d,a)$ which puts instead a congruence condition on the upper endpoint of the interval 
  $p_{k+1} \equiv a \Mod d$ and putting no congruence condition on $p_k$.  By an analysis similar to that made in Section \ref{sec:cramer}, the  modified \cramer{} model predicts 
$$\Phi^R(x;d,a) = \frac{1}{\phi(d)}x - R(d;a)\frac{x}{\log x} + o(\frac{x}{\log x}),$$  
with the bias term having the opposite sign as for the prime running function.

A more refined analysis of the biases of prime running function and its generalizations can be done based on the Hardy-Littlewood $k$-tuple conjecture, following ideas in the paper of Lemke-Oliver and \sound~\cite{LOS16}.
We leave this topic for future work.

 \section*{Acknowledgments}
The author made an initial observation similar to Figure 2
with Upamanyu Sharma, whom he thanks for the help
in the initial computation of prime running functions. The author
thanks J. C. Lagarias for mentoring this project, for advice on writing, and supplying references.
The author thanks Corey Everlove and Djordje Mili\'cevi\'c for helpful comments. This work was partially supported by NSF grant DMS-1701576.

\section{Appendix: Proof of Lemma \ref{lemma:geo_simplify}}
\begin{customlemma}{\ref{lemma:geo_simplify}}
\GeoApproxLemma
\end{customlemma}

\begin{proof}
We begin by decomposing $\abs{T_1(n) - T_2(n)}$ into two parts using the triangle inequality.
\begin{equation*}
  \abs{T_1^m(n) - T_2^m(n)} \leq H_1^m(n) + H_2^m(n)
\end{equation*}
where
\begin{equation}\label{eq:def_h1}
  H_1^m(n) := \sum_{k>0}\paren{u_{n + k} - u_n}^m\paren{\frac{1}{\log u_n} - \frac{1}{\log u_{n+k}}}\paren{1 - \frac{c}{\log u_n }}^{k-1}
\end{equation}
\begin{equation}\label{eq:def_h2}
  H_2^m(n) := \sum_{k > 1}\frac{\paren{u_{n+k} - u_n}^m}{\log u_{n + k}}\left[\prod_{j = 1}^{k-1}\paren{1 - \frac{c}{\log u_{n + j}}} - \paren{1 - \frac{c}{(\log u_n)}}^{k-1}\right]
.\end{equation}
Thus it suffices to show the following inequalities.
\begin{equation}\label{eq:approx_term_1}
  H_1^m(n) = \bigo{\frac{(\log n)^{m}}{n}}
\end{equation}
\begin{equation}\label{eq:approx_term_2}
  H_2^m(n) = \bigo{\frac{(\log n)^{m + \epsilon}}{n}}
.\end{equation}

We first prove \eqref{eq:approx_term_1}. Note that $\frac{d}{dx}\frac{1}{\log x} = -\frac{1}{x(\log x)^2}$ is decreasing in magnitude. Thus by mean value theorem,
\begin{equation}\label{eq:inverse_log_diff}
  \frac{1}{\log x} - \frac{1}{\log(x + t)} \leq \frac{t}{x(\log x)^2} \quad t \geq 0
\end{equation}
By substituting \eqref{eq:inverse_log_diff} into \eqref{eq:def_h1} we establish that
\begin{equation}\label{eq:log_diff_subbed}
  H_1^m(n) \leq \sum_{k > 0}\frac{(u_{n+k} - u_n)^{m + 1}}{u_n(\log u_n)^2}\paren{1 - \frac{c}{\log u_n}}^{k-1}
\end{equation}
Note that for any positive integer $a$, there exists some integer $u \in [a, a + Q)$ co-prime to $Q$. Thus the following holds
\begin{equation}\label{eq:unseive_bound}
  n \leq u_n \leq Qn
\end{equation}
\begin{equation}\label{eq:unseive_gap_bound}
  k \leq u_{n + k} - u_n \leq Qk \quad k = 0, 1, 2,\ldots
.\end{equation}
By substituting \eqref{eq:unseive_bound} and \eqref{eq:unseive_gap_bound} into \eqref{eq:log_diff_subbed}, we obtain that
\begin{equation}
  H_1^m(n) \leq \frac{Q^{m+1}}{u_n\log u_n}\sum_{k > 0}\frac{k^{m + 1}}{\log u_n}\paren{1 - \frac{c}{\log u_n}}^{k-1}
.\end{equation}
Let $Y_p$ be a geometric random variable with parameter $p = \frac{c}{\log u_n}$.
By substituting the definition for the $m+1\textsuperscript{th}$ moment, we obtain that
\begin{equation}\label{eq:H_1_bound_simp}
    H_1^m(n) \leq \frac{Q^{m+1}}{c u_n \log u_n} \EE\left[Y^{m+1}_p\right]
.\end{equation}

We will use the moment generating function $M(t) = \EE[\exp(t Y_p)]$ to bound the growth of $\EE\left[Y^{m+1}\right]$. 
By direct computation, $M(t) = \frac{pe^t}{1 - e^t(1 - p)}$. By utilizing the fact that $\frac{\partial^{m + 1}}{\partial t^{m + 1}}M(t) |_{t = 0} = \EE[Y_p^{m + 1}]$,
we conclude that
\begin{equation}\label{eq:geo_moment_bound}
\EE[Y^{m + 1}_p] = \bigo{\frac{1}{p^{m + 1}}}_{p \in (0, 1]}.
\end{equation}

By substituting \eqref{eq:geo_moment_bound} into \eqref{eq:H_1_bound_simp} and $p = \frac{c}{\log u_n}$, we obtain that
\begin{equation}\label{eq:approx_sub1}
    H_1^m(n)= \bigo{\frac{\log(u_n)^m}{u_n}}
.\end{equation}

Now all we have left is to prove \eqref{eq:approx_term_2}.
Note that for any $k > 0$,
\begin{equation*}
  \paren{1 - \frac{c}{\log u_n}}^{k-1} \leq \prod_{j = 1}^{k-1}\paren{1-\frac{c}{\log u_{n+j}}} \leq \paren{1 - \frac{c}{\log u_{n+k}}}^{k-1}
.\end{equation*}
Thus
\begin{equation}\label{eq:h2_bound}
  H_2^m(n) \leq \sum_{k > 0}\frac{\paren{u_{n+k} - u_n}^{m}}{\log u_{n+k}}\left[\paren{1-\frac{c}{\log u_{n+k}}}^{k-1} - \paren{1 - \frac{c}{\log u_n}}^{k-1}\right]
.\end{equation}
Note that $\frac{d}{dx} \paren{1 - \frac{c}{\log x}}^k = \frac{ck}{x(\log x)^2}\paren{1 - \frac{c}{\log x}}^{k-1}$. Thus the derivative of the function $(1 - \frac{c}{\log x})^k$ is bounded above by $\frac{ck}{a(\log a)^2}\paren{1 - \frac{c}{\log b}}^{k-1}$ over the interval $x \in [a, b]$. By mean value theorem, we establish that for $e^c < a < b$,
\begin{equation}\label{eq:log_pow_diff}
  \paren{1 - \frac{c}{\log b}}^k - \paren{1 - \frac{c}{\log a}}^k \leq \frac{ck(b - a)}{a(\log a)^2}\paren{1 - \frac{c}{\log b}}^{k-1}
.\end{equation}
By substituting \eqref{eq:log_pow_diff} into \eqref{eq:h2_bound}, we establish that for sufficiently large $n$,
\begin{equation}\label{eq:h2mvt}
  H_2^m(n) \leq c\sum_{k > 0}\frac{\paren{u_{n+k}-u_n}^{m + 1}(k-1)}{\log(u_{n+k})u_n(\log u_n)^2}\paren{1 - \frac{c}{\log u_{n+k}}}^{k-2}
.\end{equation}
By substituting \eqref{eq:unseive_bound} and \eqref{eq:unseive_gap_bound} into \eqref{eq:h2mvt}, we obtain that
\begin{equation*}
  H_2^m(n) \leq \frac{c Q^m}{n(\log n)^3}\sum_{k > 0}{k^{m+1}(k-1)}\paren{1 - \frac{c}{\log(Q(n+k))}}^{k-2},
\end{equation*}.

By noting that $(1 - \frac{c}{\log(Q(n+k))})^{-2} \leq 2$ for sufficiently large $n$, we obtain that
\begin{equation*}
  H_2^m(n) \leq 2\frac{c Q^m}{n(\log n)^3}\sum_{k > 0}{k^{m+2}}\paren{1 - \frac{c}{\log(Q(n+k))}}^{k},
\end{equation*} for sufficiently large $n$.

Because $1 - x \leq e^{-x}$ for all $x \in \cR$, we know that
\begin{equation}
  H_2^m(n) \leq 2\frac{cQ^m}{n(\log n)^3}\sum_{k > 0}{k^{m + 2}}e^{-ck / \log(Q(n + k))}
\end{equation}
for sufficiently large $n$. Let $P = \frac{m + 3}{m + 3 + \epsilon}$.
Since $P < 1$, there exists a constant $C_P > 0$ such that for all sufficiently large $n$, 
\begin{equation}
  \frac{k}{\log(Q(n + k))} \geq \frac{k^P}{C_P\log n } \quad k = 1,2,\ldots
\end{equation}
Thus for sufficiently large $n$,
\begin{equation}
  H_2^m(n) \leq 2\frac{cQ^m}{n(\log n)^3}\sum_{k > 0}{k^{m + 2}}e^{-c\frac{k^P}{C_a\log n}}
\end{equation}
\begin{equation}\label{eq:h2_bound_simp}
  H_2^m(n) \leq 2\frac{cQ^m}{n(\log n)^3}\int_0^\infty {(t + 1)^{m + 2}}e^{-c\frac{t^P}{C_P\log n}} dt
\end{equation}
By substituting $u = t^a$, we obtain that

\begin{equation}\label{eq:integral_bound}
  \int_0^\infty {(t + 1)^{m + 2}}e^{-c\frac{t^P}{C_P\log n}} dt = \bigo{(\log n)^{m + 3 + \epsilon}}
.\end{equation}
By substituting \eqref{eq:integral_bound} into \eqref{eq:h2_bound_simp}, we conclude \eqref{eq:approx_term_2}.
\end{proof}

\end{document}